\documentclass[12pt]{amsart}%
    \usepackage[T1]{fontenc}
    \usepackage[utf8]{inputenc}
    \usepackage{verbatim}
    \usepackage{latexsym}
    \usepackage{amsmath}
    \usepackage{amsfonts}
    \usepackage{tabularx}
    \usepackage{amssymb,hyperref}
    \usepackage{graphicx}%
    \usepackage[normalem]{ulem}
    \usepackage{mathtools}
    \usepackage[normalem]{ulem}

    \usepackage{dsfont}

    \newcommand{\C}{\mathbb{C}}
    
    \newcommand{\R}{\mathbb{R}}
    
    \newcommand{\Z}{\mathbb{Z}}
    
    \newcommand{\tr}{\mathrm{tr}} 
    
    \newcommand{\N}{\mathbb{N}}
    
    \newcommand{\Mod}[1]{\ (\mathrm{mod}\ #1)}
    \DeclareMathOperator{\disc}{disc}

    \newtheorem{thm}{Theorem}
    
    \newtheorem{prop}[thm]{Proposition}
    \newtheorem{coro}[thm]{Corollary}
    \newtheorem{lem}[thm]{Lemma}
    
    \newtheorem{rem}[thm]{Remark}

    \numberwithin{equation}{section}
    \numberwithin{table}{section}
    \numberwithin{figure}{section}
    \numberwithin{thm}{section}

    \newcommand{\e}{\mathrm{e}}

    \newcommand{\lcm}{\mathrm{lcm}}

    \newcommand{\re}{\mathop{\mathrm{Re}}} 
     
    \newcommand{\hypgeo}[2]{%
      {\vphantom{F}}_{#1}\kern-\scriptspace F_{#2}%
    }
    
    \usepackage{xcolor}
    
    \newcommand{\kommentar}[1]{}
    \kommentar{
    \newcommand*\pFqskip{8mu}
    \catcode`,\active
    \newcommand*\pFq{\begingroup
            \catcode`\,\active
            \def ,{\mskip\pFqskip\relax}%
            \dopFq
    }
    \catcode`\,12
    \def\dopFq#1#2#3#4#5{%
            {}_{#1}F_{#2}\left(\left.\genfrac..{0pt}{}{#3}{#4}\right|#5\right)%
            \endgroup
    }}

    \newmuskip\pFqmuskip
    \newcommand*\pFq[6][8]{%
      \begingroup 
      \pFqmuskip=#1mu\relax
      \mathchardef\normalcomma=\mathcode`,
      
      \mathcode`\,=\string"8000
      
      \begingroup\lccode`\~=`\,
      \lowercase{\endgroup\let~}\pFqcomma
      
      {}_{#2}F_{#3}{\left(\!\left.\genfrac..{0pt}{}{#4}{#5}\right|#6\!\right)}%
      \endgroup
    }
    \newcommand{\pFqcomma}{{\normalcomma}\mskip\pFqmuskip}

    \newcommand{\nedit}[1]{{\color{blue} }}
    
    \title{The Minimal Degree of Salem Numbers with Negative Trace}
    \author{Subham Roy}
    \address{Charles University \\ Faculty of Mathematics and Physics \\ Department of Algebra \\ Sokolovsk\'{a} 83 \\ 186 75 Praha 8 \\ Czech Republic}
\email{subham.roy@matfyz.cuni.cz}
    \author{Pavlo Yatsyna}
    \address{Charles University \\ Faculty of Mathematics and Physics \\ Department of Algebra \\ Sokolovsk\'{a} 83 \\ 186 75 Praha 8 \\ Czech Republic}
\email{p.yatsyna@matfyz.cuni.cz}

\subjclass[2020]{Primary 11R06; Secondary 11P32, 11C08}
\keywords{Salem numbers, Schur-Siegel-Smyth trace problem, circular interlacing, cyclotomic points, Vinogradov's theorem}
    
\begin{document}
 
    \begin{abstract}
  We show that the minimal degree of a Salem number with prescribed negative trace is bounded below by $T^2/\log T$ and above by $T^2\log T$, up to explicit constants and lower-order terms, where $T$ is the absolute value of the trace. This improves the previous doubly exponential upper bound of McKee and Smyth. Building on their construction, we also prove that for every prescribed negative trace there exists a constant $d_0$ such that Salem numbers of that trace exist in every even degree exceeding $d_0$.
    \end{abstract}
       \maketitle
    \section{Introduction}

    A \textit{Salem number} is a real algebraic integer $\tau > 1$ whose other conjugates all have the absolute value at most $1$, with at least one having modulus exactly $1$. The minimal polynomial $P(x)$ of $\tau$, a \textit{Salem polynomial}, is necessarily reciprocal (i.e. $x^{\deg P}P\left(\frac{1}{x}\right) = P(x)$), has an even degree of at least four, and all its roots other than $\tau$ and $\frac{1}{\tau}$ have modulus exactly $1$.  A comprehensive account of properties Salem numbers and their applications throughout various fields of mathematics can be found in Smyth's survey~\cite{SSurvey}.
    
    Salem numbers of any positive trace are easily constructed, for instance, the largest root of \[x^4 - nx^3 - (2n+1)x^2 - nx + 1\] is a Salem number of trace $n$ for every $n \geq 1$. McMullen~\cite[p. 230]{McM} asked whether there are Salem numbers with traces less than $-1$. The negative traces proved far more elusive, in particular not having a family of fixed degrees producing all negative traces, as in the case for positive traces. 
    
    The breakthrough came with McKee and Smyth~\cite{MS}, that \textit{there are Salem numbers for every trace}. They introduced an interlacing condition: pairs of polynomials whose roots alternate on the unit circle are combined to produce a polynomial with the prescribed negative trace, with a Salem polynomial as a factor (see Section~\ref{preinter}). Although their method produces the required Salem polynomials, the presence of other cyclotomic factors in the constructed polynomials affects the trace and degree. To overcome this issue, they eliminated possible cyclotomic factors using bounds on torsion cosets in algebraic tori (based on ideas of Schmidt~\cite{Sch}). For an odd positive integer $n$, the resulting degree bound is doubly exponential in $n$ for a Salem number with trace $-\left\lfloor\frac{n}{2}\right\rfloor$. They mentioned that the computations for the negative trace down to $-25$ indicate that the upper bound, although effective, is far from sharp~\cite[Section 8]{MS}.
    
    The minimal degrees of Salem numbers with prescribed negative traces have also been studied. When the trace is $-1$, the degree must be at least $8$, and there exist Salem numbers of every even degree greater than $8$ with this trace~\cite{S-1}. 
    
    McKee and Smyth~\cite{MS-2} showed that the minimal degree of a Salem number of trace $-2$ is exactly $20$ and that there are exactly two of them. They also solved a closely related problem: the minimal degree $d$ of a totally positive algebraic integer with a trace at most $2d - 2$ is $10$. 
    Later, McKee and Yatsyna~\cite[Proposition 1]{MY} showed that for all $d \geq 12$, there is a Salem number of degree $2d$ and trace $-2$.
    
    For trace $-3$, McKee~\cite[Section 5.4]{M11} was the first to find an example of degree $54$ using the interlacing condition for the sums of pairs of polynomials. This was later reduced to $34$ (the smallest known degree of Salem numbers of trace $-3$) by El Otmani, G. Rhin, and J.-M. Sac-\'{E}p\'{e}e~\cite{ERS}. More recently, Cherubini and Yatsyna~\cite{CY} proved that Salem numbers of trace $-3$ exist in every degree $\geq 34$, and, together with the results due to Wang, Wu and Wu~\cite{WWW}, which say that there are no Salem numbers of trace $-3$ and degree $\leq 30$, only the degree $32$ case remains undecided. Sac-\'{E}p\'{e}e~\cite{S23} has since tabulated explicit minimal polynomials of Salem numbers of trace $-3$ and degrees $34, 36, 38$, and $40$.
    
    Let $T$ be a positive integer, define
    	\begin{equation}\label{DT}
    	    D(T)=\min\{\deg\tau:\ \tr(\tau)\le -T,\,\tau\text{ is a Salem number}\}.
    	\end{equation}
    
    The McKee and Smyth results show that $D(T)\le \exp\exp (22+4T\log T)$~\cite[Theorem~1]{MS}. Using the bounds for the traces of totally positive integers, one can show that $D(T) > 10.1025T$ for sufficiently large $T$ (see \eqref{lowerbdDT} and \eqref{ineq:linlower}). We improve both of these:
    
    \begin{thm}\label{thm:uplowbd}
        For sufficiently large positive integer $T$ the following holds
        \[\frac{2T^2}{16\log (2T) + 11} \leq D(T) \leq 2(1 + o(1))T^2\log T.\]
    \end{thm}
    
    In Theorem~\ref{finalthm}, we derive an explicit upper bound of $D(T).$ The following corollary is an immediate application of the upper bound in Theorem~\ref{thm:uplowbd}.

    \begin{coro}\label{coro:totpos}
        For infinitely many $m$, there exists a totally positive algebraic integer of degree $m$ and trace less than $2m - c\sqrt{\frac{m}{\log m}}$, for some explicit constant $c > 0$.
    \end{coro}
    
     Note that this improves the upper bound given in~\cite[Corollary 1.3]{MS}, and the proof follows from the fact that, for a Salem number $\tau$, $\tau + \frac{1}{\tau} + 2$ is totally positive and the degree of the minimal polynomial of $\tau + \frac{1}{\tau} + 2$ is half of the degree of the minimal polynomial of $\tau$ (see Lemma~\ref{lem:dict}).
    
    Finally, using the results in~\cite{MS}, we show that there exists a Salem number of a prescribed trace for every sufficiently large even degree. In particular, we have the following result.
    
    \begin{thm}\label{salemlargedegree}
        For every odd integer $n \geq 10^4$, there exists a constant $d_0(n)$ such that there is a Salem number with trace $-\lfloor\frac{n}{2}\rfloor$ for every even degree $2d$ for all $d \geq d_0(n)$.
    \end{thm}
    
     The constant $d_0(n)$ can be made explicit and of size $O(\exp(\exp(n\log n)))$ (see Section~\ref{theorem103}, specifically \eqref{don}, \eqref{killerexponent} and Proposition~\ref{alllarged}).

    \bigskip
    
    \subsection{Outline of the proofs}
    
    One of the main ingredients in the proof of Theorem~\ref{thm:uplowbd} is McKee and Smyth's construction of Salem numbers using cyclotomic polynomials interlacing on the unit circle ~\cite[Section 3]{MS}. The article is structured in the following way. 
    
    Section~\ref{CISN} recalls the McKee–Smyth circular interlacing framework: a coprime pair of interlacing polynomials $A(z), B(z)$ satisfying the hypotheses of Proposition~\ref{prop:MS} yields, via $(z^2 - 1)B(z) - zA(z)$, the minimal polynomial of a Salem number, up to cyclotomic factors. In the sections that follow, we control those cyclotomic factors well enough to obtain the claimed growth of the trace and the degree.
    
    In Section~\ref{SNPnT}, we construct the two-variable polynomial $P_n(x, y)$ (see \eqref{polynomialP}), whose specialisations in $y = x^M$ interlace and have traces $-\lfloor \frac{n}{2}\rfloor + 1$. \S~3.1 contains an application of the Beukers–Smyth theorem (Theorem~\ref{thmbeukerssmyth}) to show $P_n(x, y) = 0$ has only finitely many cyclotomic points; in particular, Lemma~\ref{bdvq} bounds the relevant Newton polytope area, and Lemma~\ref{finitecycpnt} rules out the degenerate factors that would instead force infinitely many cyclotomic points. In \S~3.2, Lemmas~\ref{lem: mainploy*},~\ref{lem:traceQn} and~\ref{qnsep} isolate explicit cyclotomic factors of $P_n(x,x^M)$ and account for every contribution of additional traces beyond $-\lfloor \frac{n}{2} \rfloor$ by contributions of identifiable cyclotomic factors.
    
    In Section~\ref{CCP}, we show that the cyclotomic points that occur form a finite, effectively bounded list, giving the key bound for the Salem number trace appearing as a root of $P_n(x, x^M)$ (Lemma~\ref{realtraces}), with explicit and effective bounds for $n \geq 10^4$ given in Lemmas~\ref{boundtr},~\ref{boundL} and~\ref{sizeIr}.
    
    Section~\ref{PT1} contains a construction of a family of $k_n$ candidate exponents $M_j$ (to be specified later). The cyclotomic-factor counts $C_j$ are then averaged across the family (Lemmas~\ref{uppbdmrir},~\ref{boundsumcj} and~\ref{boundknphitr}), and a contradiction forces at least one $M_{j_0}$ with $C_{j_0} = o(n)$, establishing the claimed upper bound on the degree (Theorem~\ref{finalthm}).
    
    In Section~\ref{sec:mainproof}, we reframe the bound via the \emph{Schur--Siegel--Smyth trace problem}, improve the upper bound of $D(T)$ for sufficiently large $T$ (Proposition~\ref{prop:upper}), and obtain a lower bound of $D(T)$ showing that no linear bound is possible (Proposition~\ref{prop:main}, Corollary~\ref{cor:nolinear}), ultimately proving Theorem~\ref{thm:uplowbd}.
    
    Finally, in Section~\ref{theorem103}, for the proof of Theorem~\ref{salemlargedegree}, we adapt the McKee–Smyth construction \cite[Section 5]{MS}, which produces Salem numbers of trace $-\lfloor n/2\rfloor$ via a certain polynomial $h(x, x^{\ell_0}, \cdots, x^{\ell_n})$ with even degree, namely, $2 + \sum_{i=0}^n \ell_i - n$. The key observation is that any $n$ distinct primes $\ell_1, \ldots, \ell_n$ coprime to the \textit{killer exponent} $\ell_0$ (see \eqref{killerexponent} and \cite[Lemma 5.2]{MS}) yield a Salem number of the same trace. Using Vinogradov's three-primes theorem \cite{Vino} and density arguments (Proposition~\ref{alllarged}, Lemmas~\ref{smallprimes},~\ref{sameprimes} and~\ref{thebound}), we show that every sufficiently large odd number is representable as a sum of $n$ such primes, thereby realising all large even degrees.
    
    \bigskip
    
    \subsection{Notation}\label{notation} Throughout this paper, $\N$, $\Z$, $\R_{>0}$, $\R$, and $\C$ denote sets of positive integers, integers, positive real numbers, real numbers, and complex numbers, respectively. For a real number $\alpha$, let $\lfloor \alpha \rfloor$ and $\lceil \alpha \rceil$ denote the largest integer $\leq \alpha$ and the smallest integer $\geq \alpha$. 
    
    For functions $F: \R \to \R$ and $G: \R \to \R_{>0}$, we use the standard \emph{big $O$} and \emph{little $o$} notation: we write $F(x) = O(G(x))$ to mean $|F(x)| \le C G(x)$ for some $C > 0$, and $F(x) = o(G(x))$ to mean $\lim_{x \to \infty} F(x)/G(x) = 0$. 
    
    For $M \in \N$, $\omega(M)$ denotes the number of distinct prime factors of $M$, and $p_m$ denotes the $m$-th prime. The \emph{M\"{o}bius function} $\mu(n)$ is defined such that $\mu(1)=1$, $\mu(n)=(-1)^k$ if $n$ is a product of $k$ distinct primes, and $\mu(n)=0$ if $n$ is not square-free. We denote the \emph{$\ell$-th cyclotomic polynomial} by \begin{equation}\label{definecyclotomic}
        \Phi_\ell(x) = \prod_{\substack{\gcd(k,\ell)=1 \\ 0<k<\ell}} (x - \zeta_\ell^k),
    \end{equation} where $\zeta_\ell$ is the \emph{primitive $\ell$-th root of unity}, i.e. $\zeta_\ell=\exp(2\pi i/\ell)$.
    
    For a polynomial $P(x) = \sum_{i=0}^d a_ix^i \in \C[x]$, with $a_d \neq 0$, denote by $\deg p$  and $\tr(P)$ its \emph{degree} and \emph{trace}, which are $d$ and $ -\frac{a_{d-1}}{a_d}$ for $P(x)$, respectively.
    
    Finally, $a \mid b$ and $a \nmid b$ denote divisibility and its negation, respectively; $a^n \parallel b$ indicates that $a^n$ is the highest power of $a$ dividing $b$; and $\#\mathcal{A}$ denotes the cardinality of a finite set $\mathcal{A}$.
    
    \medskip
    
    \subsection{Inequalities involving bounds on primes}
    The following inequalities, due to Rosser and Schoenfeld~\cite[Theorem~3, Corollary, eq.~(3.13), Theorem~9]{RS} and Robin~\cite[Theorem~6]{GR}, are essential to prove our theorem. For $m \geq 6$, \begin{align}
        m\log m \, < \, p_m \, <& \, m(\log m + \log\log m),  \label{pmbound} \\ \sum_{t=1}^m p_t \, <& \, m^2(\log m + \log\log m), \label{sumpmbound} \\ 
        m(\log m + \log\log m - c_0)\, \leq \, \sum_{t=1}^m \log p_t \, <& \, 2p_m < 2m(\log m + \log \log m), \label{primorialbd} 
    \end{align} where $c_0 = 1.0769$, and the inequalities in \eqref{primorialbd} hold for all $m \geq 2$.

    \section*{Acknowledgements}
   We thank Giacomo Cherubini, Siu Hang Man, and Stelios Sachpazis for many interesting discussions on the topic. Both authors were supported by Charles University programme PRIMUS/24/SCI/010. P.Y. was supported by Czech Science Foundation, grant number 26-20514S.
    
    \section{Circular Interlacing and Salem numbers}\label{CISN}
    
    Before proceeding with the proof of Theorem~\ref{thm:uplowbd}, we very briefly recall the relation between the interlacing of polynomials and the construction of Salem numbers.
    
    \smallskip
    
    \subsection{Circular Interlacing and Salem numbers}\label{preinter}
    
    Two polynomials $A(z)$ and $B(z)$ with integer coefficients are said to satisfy the \textit{circular interlacing} condition if they are coprime, have positive leading coefficients, and all their simple roots alternate (interlace) along the unit circle~\cite[Section 2.2]{MSinter}. 
    
    By applying a transformation such as $x = \sqrt{z} + 1/\sqrt{z}$ or $x = z + \frac{1}{z}$, the interlacing roots on the unit circle map to the real line, which yields a corresponding real interlacing quotient $\alpha(x)/\beta(x)$ with strictly interlacing real zeros~\cite{MS, MSinter}. If $A/B$ is a circular interlacing quotient with $B$ monic, and the real quotient satisfies the limit condition $\lim_{x\rightarrow2+}\frac{\alpha(x)}{\beta(x)} > 2$, then only non-zero solutions to the rational equation \[\frac{A(z)}{(z-1)B(z)} = 1 + \frac{1}{z}\] are a Salem number (or a reciprocal quadratic Pisot number), its conjugates, and possibly some roots of unity (see~\cite[Theorem~3.1]{MSinter}). In~\cite[Table 1]{MSpairs}, McKee and Smyth compute an exhaustive list of (primitive) pairs of cyclotomic polynomials $(A(z), B(z))$ such that $A/B$ satisfies circular interlacing. We are specifically interested in the family where \begin{equation}\label{thefamily}
        A(z) = \frac{z^{m+1}-1}{z-1}, \, B(z) = \frac{(z^j - 1)(z^{m+1 - j}-1)}{z-1}, \, \text{} \, m \geq 1, 1 \leq j \leq \frac{m+1}{2}, \, \text{and} \, \gcd(m+1, j) = 1.
    \end{equation}

     Alternatively, sums of interlacing cyclotomic polynomial fractions, such as \[\frac{A(z)}{B(z)} = \frac{z^{p_1+p_2}-1}{(z^{p_1}-1)(z^{p_2}-1)} + \dots + \frac{z^{p_{2T-1}+p_{2T}}-1}{(z^{p_{2T-1}}-1)(z^{p_{2T}}-1)},\] also yields polynomials $A(z)$ and $B(z)$ that satisfy the interlacing condition~\cite{MS, MSpairs, MY, CY}, where $p_t$ is the $t$-th prime. Clearing the denominators in these cases, the algebraic expression $(z^2-1)B(z) - zA(z)$ successfully isolates the minimal polynomial of a Salem number. This method has been crucial in discovering infinite families of Salem numbers of specified negative traces and degrees~\cite{MS, MY, CY}. In particular, 
     \begin{prop}\cite[Proposition~3.2~(a)]{MS}\label{prop:MS}
         Let $A,B\in \Z[z]$ satisfy the circular interlacing condition, with $B$ being monic. If $B(1)=0$, or $A(1)=0$ and $2B(1)-A'(1)<0$, then $(z^2-1)B(z)-zA(z)$ is the minimal polynomial of a Salem number (or a reciprocal quadratic Pisot number), possibly multiplied by cyclotomic polynomials.
     \end{prop}
    
     In the following sections, we construct our polynomial using circular interlacing with prescribed traces and investigate properties related to its possible cyclotomic factors.

    \section{Salem numbers, the polynomial \texorpdfstring{$P_n$}{P\_n} and traces}\label{SNPnT}
    
    \subsection{A theorem of Beukers and Smyth}
    
    Following the construction in~\cite[Lemma 6]{MY}, we consider the polynomial, for an odd positive integer $n$, \begin{equation}\label{mainploq}
        P_n(x, y) := \frac{R_n(x,y)x(y-1)\prod^n_{i=1}(x^{p_i}-1)}{(x-1)^{n-1}}\in\Z[x,y],
    \end{equation} where 
    \begin{equation}\label{mainpolynew}
        R_n(x, y) =  \frac{x^2 - 1}{x} -\left(\frac{x^{p_1 + p_2} - 1}{(x^{p_1}-1)(x^{p_2}-1)} + \frac{x^{p_3 + p_4} - 1}{(x^{p_3}-1)(x^{p_4}-1)} + \cdots + \frac{x^{p_n}y - 1}{(x^{p_n}-1)(y-1)}\right),
    \end{equation} 
    and $p_t$ is the $t$-th prime for $t = 1, \dots, n$. After multiplying the factors in \eqref{mainploq}, we obtain the following expression of $P_n(x, y)$: \begin{align}
    &(x-1)^{n-1}P_n(x, y) \nonumber \\ =& \, y\left[(x^2-1)\prod_{\ell = 1}^n (x^{p_{\ell}} - 1) - x\sum_{t=1}^{\lfloor \frac{n}{2} \rfloor} (x^{p_{2t-1} + p_{2t}} - 1)\prod_{\substack{\ell = 1 \\ \ell \neq 2t-1, 2t}}^n (x^{p_\ell} - 1) - x^{p_n +1} \prod_{\ell = 1}^{n-1}(x^{p_{\ell}} - 1)\right] \nonumber \\ &- \left[(x^2-1)\prod_{\ell = 1}^n (x^{p_{\ell}} - 1) - x\sum_{t=1}^{\lfloor \frac{n}{2} \rfloor} (x^{p_{2t-1} + p_{2t}} - 1)\prod_{\substack{\ell = 1 \\ \ell \neq 2t-1, 2t}}^n (x^{p_\ell} - 1) - x \prod_{\ell = 1}^{n-1}(x^{p_{\ell}} - 1)\right] \label{polynomialP}  \\ =& \, (x-1)^{n-1}(yf(x) - g(x)), \label{relftog}\end{align}   
    where $f,g\in \Z[x]$.
    
    For every $n,M\in \N$, where $n$ is odd, and $M>1$, the above expression implies that \[\deg P_n(x, x^M) = M + 2 + \sum_{t=1}^n p_t - (n-1)\] and the trace of the polynomial $P_n(x, x^M)$ (see Section~\ref{notation} for the definition) is \begin{align}
        & \, \tr(P_n(x, x^M)) + n-1 = \tr(P_n(x, x^M)) + \tr((x-1)^{n-1})  = \tr((x-1)^{n-1}P_n(x, x^M)) = \left\lfloor\frac{n}{2}\right\rfloor + 1 \nonumber \\ \Rightarrow & \,\tr(P_n(x, x^M)) = -\left\lfloor\frac{n}{2}\right\rfloor + 1. \label{traceofpn}
    \end{align} Indeed, since \eqref{polynomialP} and \eqref{relftog} imply $\deg f = \deg g$, the degree of $x^M f(x)$ is strictly greater than the degree of $g$ for all $M>1$. Therefore, \[\tr(P_n(x, x^M)) = \tr(x^M f(x) - g(x)) = \tr(x^Mf(x)) = \tr(f(x)),\] computing the trace of $f$ from~\eqref{polynomialP} and~\eqref{relftog}, we obtain the expression of the trace in~\eqref{traceofpn}. 
    
     For a polynomial $H(x, y) \in \C[x, y],$ we denote by $V(H)$ the area of the Newton polytope of $H$. We need the following theorem of Beukers and Smyth~\cite{BS} to bound the number of cyclotomic points on the curve defined by $P_n(x, y)=0$.

    \begin{thm}\cite[Theorem~4.1]{BS}\label{thmbeukerssmyth}
        Let $H \in \C[x, x^{-1}, y, y^{-1}]$, having Newton polytope of area $V(H)$. Then $H$ has either at most $22V(H)$ cyclotomic points or infinitely many. In the latter case, $H$ has a factor $x^iy^j - \zeta$ for some root of unity $\zeta$ and some integers $i, j$ not both $0$. 
    
    \end{thm}
    
    The goal is to consider the intersection of $P_n(x, y) = 0$ with the curves $y = x^{M_j}$ for $j = 1, \dots, k_n$ (where $k_n$ and $M_j$ are carefully chosen such that $M_j = O(n^2\log n)$), and to show that at least one of the $P_n(x, x^{M_j})=0$ has at most $o(n)$ cyclotomic points, and therefore, following the discussion in Section~\ref{preinter}, it produces a Salem polynomial of trace $-\lfloor\frac{n}{2}\rfloor + o(n)$.
    
    Let $\deg_y{H}$ and $\deg_x{H}$ denote the degree of $H(x, y)$ as a polynomial in $y$ and in $x$, respectively. Then we have the following result. 
    
    \begin{lem}\label{bdvq}
        The area of the Newton polytope of the polynomial $P_n(x, y)$, namely $V(P_n)$, is bounded above by $3-n + \sum_{t=1}^n p_t$, where $p_t$ is the $t$-th prime.
    \end{lem}
    
    \begin{proof}
        The Newton polytope of $P_n$ is contained in the convex hull of the vertices $(0, 0), (\deg_x P_n, 0)$, $(0, \deg_yP_n)$, and $(\deg_xP_n, \deg_yP_n)$. The area of this convex hull (i.e. the rectangle defined by these vertices) is $ (\deg_xP_n)(\deg_yP_n)$, and therefore \[V(P_n) \leq (\deg_xP_n)(\deg_yP_n).\] Since $\deg_xP_n = 2 + \left(\sum_{t=1}^n p_t\right) - (n-1)$ and $\deg_yP_n = 1$ (see \eqref{polynomialP}), we have \begin{equation}\label{bdonarea}
        V(P_n) \leq (\deg_xP_n)(\deg_yP_n) = 3-n + \sum_{t=1}^n p_t,
    \end{equation} which completes our proof.
    \end{proof}
    
    The next lemma ensures that $P_n(x, y)$ satisfies the conditions of Theorem~\ref{thmbeukerssmyth}, and therefore the number of cyclotomic points of $P_n$ is bounded above by the area of the Newton polytope of $P_n$, i.e. $V(P_n)$. 
    
     \begin{lem}\label{finitecycpnt}
           For $n \geq 3$, $P_n(x, y)$ has no factor of the form $x^iy^j - \zeta$ for some root of unity $\zeta$ and some integers $i, j$ not both $0$.
        \end{lem}
    
        \begin{proof}
            From the definition of $P_n(x, y)$ with \eqref{polynomialP}, it follows that $P_n(x, y)$ is not of the form $x^iy^j - \zeta$ for some root of unity $\zeta$ and some integers $i, j$ not both $0$.
            
            Since $P_n(0, y) = y - 1$, this implies that $x^iy^j - \zeta$ cannot be a factor of $P_n$ for $i, j > 0$. Indeed, if, for $i, j > 0$, $P_n(x, y) = (x^iy^j - \zeta)\sum_{\ell \geq 0} d_\ell x^\ell$, where $d_{\ell} \in \C$, then the coefficient of $y$ in $P_n(0, y)$ is $0$, which is a contradiction. Therefore, we only need to check the cases where $i = 0$ or $j = 0$, but not both. 
    
            Noting that $\deg_y P_n = 1$, any factor $x^iy^j-\zeta$ with $j>0$ must have $j=1$, and its cofactor must depend only on $x$. Thus, when $i >0$, $P_n(x,y)=(x^iy-\zeta)Q(x)$, for some $Q(x)\in\C[x]$. Evaluating at $x=0$ yields $P_n(0,y)=-\zeta Q(0)$, which is constant in $y$, contradicting $P_n(0,y)=y-1$.

 Next, we consider the case where $i =0$. We have already established above that $j \leq 1.$ Now, if $y - \zeta$ is a factor of $P_n$, then \[P_n(x, y) = (y-\zeta)(\sum_{\ell >0}c_\ell x^\ell + \zeta^{-1}) = y\sum_{\ell>0}c_\ell  x^\ell - \zeta \sum_{\ell>0} c_\ell x^\ell + \zeta^{-1}y - 1,\] and since $P_n(0, y) = y-1$, for the above to hold, we need $\zeta = 1$. In this case, \eqref{polynomialP} yields $P_n(x, 1) = -(x^{p_n + 1} - x)\prod_{t=1}^{n-1}\Phi_{p_t}(x)$, which is not identically $0$, where $\Phi_{p_t}$ is the $p_t$-th cyclotomic polynomial defined in Section \ref{notation}. Therefore, $P_n$ has no factor of the form $y - \zeta$.
    
             It remains to check whether $x^i-\zeta$ is a factor of $P_n$ for some $i > 0$. Since $(x-\xi)$ divides $(x^i - \zeta)$ in $\C[x]$ for some root of unity $\xi$ such that $\xi^i = \zeta$, we note that if $x^i - \zeta$ is a factor of $P_n$ then so is $x - \xi$. Therefore, it is enough to consider whether $x-\zeta$ is a factor of $P_n$ for some root of unity $\zeta$.
            
            First recall that $n$ is odd and $P_n(x, y) = yf(x) - g(x)$, where $f$ and $g$ are as in \eqref{relftog}. Therefore, if $P_n(\zeta, y) = 0$ for all $y \in \C$, then $f(\zeta) = g(\zeta) = 0$. Consider $h(x):= g(x) - f(x)$, then it implies $h(\zeta) = 0$. Since \begin{equation}\label{f-g}
               h(x) = g(x) - f(x) = x(x^{p_n} - 1)\prod_{t =1}^{n-1}\Phi_{p_t}(x)
            \end{equation} and $\zeta$ is a root of unity, we have two types of choices for $\zeta$, namely, $\zeta = 1$ or $\zeta = \zeta_{p_t},$ the $p_t$-th root of unity, for $t = 1, \dots, n$. Evaluating $P_n(x, y)$ at $x = 1 \, \text{or} \, \zeta_{p_t}$ we find \begin{align}\label{eq:p(1,y)}
                P_n(1, y) =& -(y-1)\left(\sum_{t=1}^{\lfloor \frac{n}{2} \rfloor} (p_{2t-1} + p_{2t})\prod_{\substack{\ell = 1 \\ \ell \neq 2t-1, 2t}}^n p_{\ell} + \prod_{\ell = 1}^{n-1} p_{\ell}\right),
            \end{align} which is non-zero for $n \geq 3$. Furthermore, from \eqref{polynomialP}, we have, for $t = 1, \dots, n-1$, there exists $1 \leq u \leq \lfloor \frac{n}{2} \rfloor$ such that $p_t \in \{p_{2u -1}, p_{2u}\}$ and \begin{equation}\label{evalzetapt}
                P_n(\zeta_{p_t}, y) = -(y-1)\zeta_{p_{t}}\prod_{\substack{d \mid p_{2u-1} + p_{2u}\\d>1}} \Phi_{d}(\zeta_{p_t}) \prod_{\substack{\ell = 1 \\ \ell \neq 2u-1, 2u}}^n \Phi_{p_{\ell}}(\zeta_{p_t}).
            \end{equation} Since $p_t\nmid d$ for all divisors $d$ of $p_{2u-1}+p_{2u}$, $\Phi_d(\zeta_{p_t})$ is non-zero. Moreover, since $\Phi_{p_t}$ does not divide $\Phi_{p_\ell}$ for $\ell\notin\{2u-1,2u\}$, $\Phi_{p_\ell}(\zeta_{p_t}) \neq 0$. These imply  that $P_n(\zeta_{p_t}, y)$  is strictly non-zero when evaluated at $y \neq 1$. For $t = n$, all summation terms in $f(\zeta_{p_n})$ vanish, leaving \[P_n(\zeta_{p_n}, y) = -(y-1)\zeta_{p_n}\prod_{\ell = 1}^{n-1}\Phi_{p_\ell}(\zeta_{p_n}),\] which is also non-zero for all $y \neq 1$. This implies that $x-\zeta$ is also not a factor of $P_n(x, y)$ for any roots of unity $\zeta$, and this completes the proof.  
            \end{proof}

        \begin{rem}\label{zerooff}
 Note that since $f(x)$ and $g(x)$ in \eqref{relftog} are polynomials with integer coefficients, if $f(\zeta) = 0$ for some $\zeta \in \mu_{\infty}$ (where $\mu_{\infty}$ denotes the set of all roots of unity), then $f(\zeta^{-1}) = f(\bar{\zeta}) = g(\zeta) = 0$ and vice versa. Therefore, \[\{x \in \C: f(x) = 0\} \cap \mu_{\infty} = \{x \in \C: g(x) = 0\} \cap \mu_{\infty},\] and equation~\eqref{f-g} together with the discussion following, imply that \[\{x \in \C: f(x) = 0\} \cap \mu_{\infty} = \varnothing.\]
         \end{rem}
    
    \subsection{Salem number and \texorpdfstring{$P_n$}{P\_n}}\label{sec:tracesofqn} 
    For $M \in \N$, we define $\omega_n^*(M)$ as follows \begin{equation}\label{modifiedomega}
         \omega_n^*(M) := \left\{\begin{array}{cc}
            \#\{p : p \mid M, \,  p \leq p_{n-1}\}  & \, \text{when $M$ is not a prime}, \\
             0 & \, \text{when $M$ is a prime}.
         \end{array}\right.
     \end{equation} Therefore, $\omega_n^*(M) \leq \omega(M)$. For $M \in \N$, our goal is to investigate how the cyclotomic factors of $P_n(x, x^M)$ affect the trace. We focus on $M > p_n$ such that $\gcd(M, p_n) = 1$.
    
     $P_n(x, x^M)$ is divisible by a Salem polynomial. Furthermore, some of the cyclotomic factors of $P_n(x, x^M)$ can be explicitly determined. These factors are collected in the following lemma.

        \begin{lem}\label{lem: mainploy*}
            Let $M >p_n$ with $\gcd(M,p_n)=1$, polynomial $P_n(x,x^M)$ factors over $\Z[x]$ as
            \begin{equation}\label{mainploy*}
                 P_n(x, x^M) = (x-1)\left(\prod_{\substack{p \mid M\\ p \leq p_{n-1}} } \Phi_p(x)\right)Q_n(x, x^M),
            \end{equation}
            where $Q_n(x,x^M)\in \Z[x]$ and $\Phi_p (x)$ is the $p$-th cyclotomic polynomial.
        \end{lem}
        \begin{proof}
           From \eqref{eq:p(1,y)} it follows that $(y-1)\mid P_n(1,y)$, and therefore $P_n(1, 1) = 0$. Since $P_n(1, 1) = 0$ and $P_n(x,x^M)$ is a polynomial in $x$, it follows that $(x-1)\mid P_n(x,x^M)$.
    
           For every $p\mid M$ with $p\le p_{n-1}$, we have $P_n(\zeta_p,\zeta_p^M) = P_n(\zeta_p,1)$. Since $p \leq p_{n-1}$, \eqref{evalzetapt} yields $P_n(\zeta_p, y)$ is a product of $y-1$ and a non-zero constant obtained by evaluating $\Phi_d(x)$ at $x = \zeta_p$ (for $p \nmid d$) multiplied by $\zeta_p$, we have that $P_n(\zeta_p,1)=0$. Moreover, $P_n(x, x^M) \in \Z[x]$, and hence $P_n(\zeta_p, \zeta_p^M) = 0$ implies that $P_n(x, x^M)$ is divisible by the minimal polynomial of $\zeta_p$, that is, $\Phi_p(x) \mid P_n(x, x^M)$. 
    
           Noting that $(x-1)$ and each $\Phi_p(x)$ (for all $p \mid M$ and $p \leq p_{n-1}$) are pairwise coprime, this completes the proof. 
        \end{proof}
        Therefore, it is enough to investigate the polynomial $Q_n(x, x^M)$ for $M> p_n$ and coprime to $p_n$. We can show the following about the trace of $Q_n$. 
        \begin{lem}\label{lem:traceQn}
         For any integer $M> p_n $ such that $\gcd(M, p_n) = 1$, \begin{equation}\label{traceofqn}
          \tr(Q_n(x, x^M)) = -\left\lfloor\frac{n}{2}\right\rfloor + \omega_n^*(M).
         \end{equation}
     \end{lem}
    
     \begin{proof}
    Taking traces on both sides of \eqref{mainploy*} gives \[\tr(P_n(x, x^M)) = \tr(Q_n(x, x^M)) + \tr(x-1) + \sum_{\substack{p \mid M\\ p \leq p_{n-1}}}\tr\left(\Phi_p(x)\right) = \tr(Q_{n}(x, x^M)) + 1 + \sum_{\substack{p \mid M\\ p \leq p_{n-1}}} (-1),\] where we use the fact that the trace of the $m$-th cyclotomic polynomial is $\mu(m)$ (see~\cite[Exercise 2.14(b)]{TA}), so that $\tr(\Phi_p(x)) = \mu(p) = -1$ for a prime $p$, and, since $\tr(P_n(x, x^M)) = -\left\lfloor\frac{n}{2}\right\rfloor + 1$ (see \eqref{traceofpn}) and \eqref{modifiedomega}, we have \[\tr(Q_n(x, x^M)) = -\left\lfloor\frac{n}{2}\right\rfloor + \omega_n^*(M).\qedhere\]
      \end{proof}
    
    Following the discussion in Section~\ref{preinter} and~\cite[Proposition 3.3]{MS}, we infer that for a given $M$ such that $M> p_n$ and $\gcd(M, p_n) = 1$, we have two coprime polynomials $a_{n, M}(x)$ and $b_{n, M}(x)$ such that \begin{equation}\label{coprimeratio}
        \frac{x^{p_1 + p_2} - 1}{(x^{p_1}-1)(x^{p_2}-1)} + \frac{x^{p_3 + p_4} - 1}{(x^{p_3}-1)(x^{p_4}-1)} + \cdots + \frac{x^{p_n + M} - 1}{(x^{p_n}-1)(x^M-1)} = \frac{a_{n, M}(x)}{b_{n, M}(x)},
    \end{equation} whose roots are interlaced on the unit circle. A further observation reveals, if \[A_n(x, x^M) := (x^M-1)\sum_{t=1}^{\lfloor \frac{n}{2} \rfloor} (x^{p_{2t-1} + p_{2t}} - 1)\prod_{\substack{\ell = 1 \\ \ell \neq 2t-1, 2t}}^n (x^{p_\ell} - 1) + (x^{p_n+M} - 1)\prod_{\ell = 1}^{n-1}(x^{p_{\ell}} - 1)\] and \[B_{n}(x, x^M) := (x^M-1)\prod_{\ell = 1}^n (x^{p_{\ell}} - 1)\] are two integer coefficient polynomials, then \begin{equation}\label{altrepresentation}
        \frac{x^2 - 1}{x} - R_n(x, x^M) = \frac{A_n(x, x^M)}{B_n(x, x^M)} = \frac{a_{n, M}(x)}{b_{n, M}(x)},
    \end{equation} and $(x-1)^{n+1} \mid B_n(x, x^M)$. Furthermore, $(x-1)^n \parallel A_{n}(x,x^M)$, since \[\frac{A_n(x, x^M)}{(x-1)^n} = \frac{x^M-1}{x-1}\sum_{t=1}^{\lfloor \frac{n}{2} \rfloor} \frac{x^{p_{2t-1} + p_{2t}} - 1}{x-1}\prod_{\substack{\ell = 1 \\ \ell \neq 2t-1, 2t}}^n \Phi_{p_\ell}(x) + \frac{x^{p_n+M} - 1}{x-1}\prod_{\ell = 1}^{n-1}\Phi_{p_\ell}(x)\] evaluates to a sum of positive integers at $x = 1$. Therefore, we have $b_{n, M}(1) = 0$, and~\cite[Proposition 3.2(a)]{MS} then implies that the polynomial  \begin{equation}\label{mainpoly**}
         q_{n, M}(x) := (x^2 - 1)b_{n, M}(x) - xa_{n, M}(x).
    \end{equation} has a minimal polynomial of a Salem number, say $\tau_M$, as a factor (and all the other factors are cyclotomic). From \eqref{mainploq}, \eqref{mainpolynew} and Lemma~\ref{lem: mainploy*}, for some $c_{n, M}(x) \in \Z[x]$, we obtain the following. \begin{equation}\label{cnNx}
        Q_n(x, x^M) = c_{n, M}(x)q_{n, M}(x), \quad \text{and} \quad Q_n(\tau_M, \tau_M^M) = 0.
    \end{equation} Since, in general $M$ that are coprime to $p_n$ and larger than $p_n$, expressing the factor $\frac{Q_n(x, x^M)}{q_{n, M}(x)} \in \Z[x]$ explicitly is difficult, we investigate the polynomial $Q_n(x, x^M)$ instead of $q_{n, M}(x)$. In our next lemma, we show that the factor $\frac{Q_n(x, x^M)}{q_{n, M}(x)}$ consists only of cyclotomic factors, and thus their influence on the computation of the trace of $\tau_M$ is explicit and computable. 
    
    \begin{lem}\label{qnsep}
        For $M> p_n$ and $\gcd(M, p_n) = 1$, the roots of $Q_n(x, x^M)$ are either a Salem number or its conjugates or some roots of unity. Moreover, for any $d > 1$, $\Phi_d^3(x)$ does not divide $Q_n(x, x^M)$.
    \end{lem}
    
    \begin{proof}
        From Proposition~\ref{prop:MS}, it follows that all the roots of $q_{n, M}(x)$  are distinct (that is, $q_{n, M}(x)$ is separable) and are either a Salem number or its conjugates, or some roots of unity as its roots. We claim that \[c_{n, M}(x) = \frac{Q_n(x, x^M)}{q_{n, M}(x)} \in \Z[x]\] is also separable and has only cyclotomic factors for all such $M$. Recall that $a_{n, M}(x)$ and $b_{n, M}(x)$ are coprime and, from \eqref{coprimeratio}, we have $b_{n, M}(x) \mid B_{n}(x, x^M)$ and $a_{n, M}(x) \mid A_{n}(x, x^M)$. Let $d_{n, M}(x) \in \Z[x]$ such that \[\gcd(A_n(x, x^M), B_{n}(x, x^M)) = (x-1)^n\left(\prod_{\substack{p \mid M\\ p \leq p_{n-1}} } \Phi_p(x)\right)d_{n, M}(x).\] Using Lemma~\ref{lem: mainploy*} we obtain \begin{align*}
             (x-1)^n\left(\prod_{\substack{p \mid M\\ p \leq p_{n-1}} } \Phi_p(x)\right) Q_n(x, x^M) =& \; (x-1)^{n-1}P_n(x, x^M) \\ =& \;(x^2 - 1)B_{n}(x,x^M) - xA_{n}(x,x^M) \\ =& \; (x-1)^n\left(\prod_{\substack{p \mid M\\ p \leq p_{n-1}} } \Phi_p(x)\right)d_{n, M}(x) ((x^2 - 1)b_{n, M}(x) - xa_{n, M}(x)) \\ =& \; (x-1)^n\left(\prod_{\substack{p \mid M\\ p \leq p_{n-1}} } \Phi_p(x)\right)d_{n, M}(x) q_{n, M}(x),
        \end{align*} where the second equality follows from a combination of \eqref{mainploq}, \eqref{polynomialP} and \eqref{altrepresentation}, the penultimate equality follows from the definition of $a_{n, M}(x)$ and $b_{n, M}(x)$ with \eqref{altrepresentation}, and finally we use \eqref{mainpoly**} to obtain the last equality. The conclusion following the above series of equalities is \[d_{n, M}(x) = c_{n, M}(x).\] Noting that $c_{n, M}(x)$ is a factor of $B_{n}(x, x^M)$, which is itself a product of cyclotomic polynomials, it follows that $c_{n, M}(x)$ is a product of cyclotomic polynomials. 
        
        Since \[\frac{B_{n}(x,x^M)}{(x-1)^n\left(\prod_{\substack{p \mid M\\ p \leq p_{n-1}} } \Phi_p(x)\right)} \in \Z[x]\] is separable and $c_{n, M}(x)$ divides the quotient, the separability of $c_{n, M}$ follows. Consequently, $\gcd$ of $c_{n, M}(x)$ and $q_{n, M}(x)$ is either $1$ or a product of distinct cyclotomic polynomials. 
        
        For $d > 1$ such that $\Phi_{d}(x) \nmid c_{n, M}(x)$ and $\Phi_d(x) \nmid q_{n, M}(x)$, it trivially follows that $\Phi_{d}(x) \nmid Q_n(x, x^M)$. For all $d > 1$, such that $\Phi_{d}(x)$ divides one of $c_{n, M}(x)$ or $q_{n, M}(x)$, but not both, we have \[\Phi_{d}(x) \parallel Q_n(x, x^M).\] For the remaining $d > 1$ such that $\Phi_{d}(x) \mid \gcd(c_{n, M}(x), q_{n, M}(x))$, we have $\Phi_d^2(x) \parallel Q_n(x, x^M)$, which follows from the separability of $c_{n, M}(x)$ and $q_{n, M}(x)$, and this completes our proof.  
    \end{proof}
    
    \medskip
    
    \section{Counting cyclotomic points}\label{CCP}
    In the previous section, we showed that the number of cyclotomic points on $P_n(x, y)=0$ is finite. In this section, we study some properties of these cyclotomic points.
    
    \subsection{Types of cyclotomic points} 
    
    Let \[\mathcal{C}_n := \{(\zeta_q^s, \zeta_r^t) \in \C^* \times \C^* \mid \zeta_q, \zeta_r \in \mu_{\infty}, \, 0 < s \leq q, \, 0 < t \leq r, \, q, r \in \N, \, P_n(\zeta_{q}^s, \zeta_r^t) = 0\}\] denote the set of cyclotomic points of $P_n(x,y)=0$. Then the finiteness of $\mathcal{C}_n$ follows from Theorem~\ref{thmbeukerssmyth} and Lemma~\ref{finitecycpnt}, that is, \begin{equation}\label{boundcylpt}
         \# \mathcal{C}_n \leq 22V(P_n) \leq 22(3- n + \sum_{t=1}^n p_t) < 22n^2(\log n + \log\log n),
     \end{equation} 
     where the penultimate inequality follows from \eqref{bdonarea}, and the last inequality holds for all $n \geq 6$ (see~\eqref{sumpmbound}).

    Since we are interested in the cyclotomic points on $P_n(x, x^M)=0$, for different values of $M$, it suffices to consider the subset $\mathcal{C}_n^*$ of $\mathcal{C}_n$ defined by \[\{(\zeta_q^s, \zeta_q^t) \in \mathcal{C}_n \mid \, q \in \N, \, 0 < s , t \leq q, \, \gcd(s, q) = 1, \, t \equiv Ms \Mod q, \, \text{for some} \, M \in \N\}.\] Furthermore, if $(\zeta_q^s, \zeta_q^t) \in \mathcal{C}_n^*$, then $\Phi_q(x) \mid P_n(x, x^M)$. Therefore, for some $M \in \N$, if $(\zeta_q^{s'}, \zeta_q^{Ms'}) \in \mathcal{C}_n^*$ for some $0 < s' < q$ with $\gcd(s', q) = 1$, then $(\zeta_q, \zeta_q^{M}) \in \mathcal{C}_n^*$ and so are $(\zeta_q^s, \zeta_q^{Ms})$ for all $0 < s < q$ that are coprime to $q$. For $q \in \N$, we collect all such points in the set $\mathcal{C}_{n, q}^{*} \subseteq \mathcal{C}_n^*$, namely, \[\mathcal{C}_{n, q}^* = \{(\zeta_q, \zeta_q^t) \in \mathcal{C}_n^* \mid \, 0 < t \leq q, \, P_n(\zeta_q, \zeta_q^t) = 0\}.\] 
    
    \begin{lem}\label{cardcnq}
        For all $q \in \N$, $\# \mathcal{C}_{n, q}^* \leq 1$.
    \end{lem}
    
    \begin{proof}
        For all $q \in \N$, such that $(\zeta_q^s, \zeta_q^t) \not\in \mathcal{C}_n^*$ for any $s, t \in \Z$, the set $\mathcal{C}_{n, q}^*$ is empty by construction. Therefore, for those $q$, $\# \mathcal{C}_n^* = 0$.
    
        For the remaining $q$ such that $q \not\in \{p_1, \dots, p_n\}$, suppose $(\zeta_q, \zeta_q^{t_1}), (\zeta_q, \zeta_q^{t_2}) \in \mathcal{C}_{n, q}^*$, then, since $P_n(x, y) = yf(x) - g(x)$, \[P_n(\zeta_q, \zeta_q^{t_1}) = 0 = P_n(\zeta_q, \zeta_q^{t_2}) \Rightarrow \zeta_q^{t_1}f(\zeta_q) = g(\zeta_q) = \zeta_q^{t_2}f(\zeta_q) \Rightarrow t_1 \equiv t_2 \Mod q,\] where the last implication follows from $q$ not being any of the first $n$-primes and Remark~\ref{zerooff} (which shows that $f(\zeta_q) \neq 0$ for all such $q$). Since $0 < t_i \leq q$, we have $t_1 = t_2$, and therefore $\# \mathcal{C}_{n, q}^* = 1$.
    
        For $q = p_1, \dots, p_n$, again the proof of Lemma~\ref{finitecycpnt} and Remark~\ref{zerooff} imply that \[P_n(\zeta_q, \zeta_q^t) = 0 \Longrightarrow \zeta_q^t = 1 \Longrightarrow q \mid t,\] where the first implication follows from \eqref{evalzetapt}, as $P_n(\zeta_{q}, y) $ vanishes only when $y = 1$ for $q \in \{p_1, \dots, p_n\}$. Moreover, since $0 < t \leq q$, we have $t = q$. In these cases, we have $\# \mathcal{C}_{n, q}^* = 1$, and this completes the proof.
    \end{proof}
    
    Based on the discussion in Section~\ref{sec:tracesofqn}, we focus on the polynomials $Q_n(x, x^M)$ for $M > 1$ instead of $P_n(x, x^M)$. Moreover, we have $Q_n(1, 1) \neq 0$ for any $M > 1$, and therefore we can restrict ourselves to all $q > 1$ such that $\# \mathcal{C}_{n, q}^* = 1$. 
    
    Since $\# \mathcal{C}_n^* < \infty$, let $L_n$ denote the number of positive integers $q > 1$ with $\# \mathcal{C}_{n, q}^* = 1$, and these $q$ can be ordered as $S_1 < \cdots < S_{L_n}$.  Since $Q_n(x, x^M) \in \Z[x]$, $Q_n(\zeta_{S_r}, \zeta_{S_r}^M) = 0$ implies that $Q_n(x, x^M)$ has $\Phi_{S_r}(x)$ as a factor.
    
    For integers $M > p_n$ such that $\gcd(M, p_n) = 1$, \begin{equation}\label{factorQ}
        Q_n(x, x^M) = P_{\tau_M}(x) \times \prod_{i=1}^{L_n} \Phi_{S_i}^{m_i}(x),
    \end{equation} where $P_{\tau_M}$ is the minimal polynomial of the Salem number $\tau_M$ such that $Q_n(\tau_M, \tau_M^M) = 0$ (the existence of such $\tau_M$ from the construction), and \begin{equation}\label{valuemulti}
        m_i = \begin{cases}
        0 & \quad \mbox{when} \, \Phi_{S_i}(x) \nmid Q_n(x,x^M),\\
             1   & \quad \mbox{when} \, \Phi_{S_i}(x)\mid  Q_n(x,x^M)\,\text{and}\,\Phi_{S_i}(x) \nmid \gcd(c_{n, M}(x), q_{n, M}(x)), \\
      2   & \quad \mbox{when} \, \Phi_{S_i}(x) \mid \gcd(c_{n, M}(x), q_{n, M}(x)).
        \end{cases}
        \end{equation}
     Here $c_{n, M}(x)$ and $q_{n, M}(x)$ are as defined in~\eqref{mainpoly**} and~\eqref{cnNx}, and the (multiplicity) value of $m_i$ follows from Lemma~\ref{qnsep}.
    
    The following result bounds the trace of $Q_n(x, x^M)$ for all positive integers $M$ coprime to $p_n$.
    
    \begin{lem}\label{realtraces}
        For $M \in \N$ such that $\gcd(M, p_n) = 1$ and $M > p_n$, \[\left|\tr(\tau_M) - \tr(Q_n(x, x^M))\right| \leq 2\#\{S:\Phi_S(x)\mid Q_n(x,x^M)\},\] where $\tr(P_{\tau_M}(x)) = \tr(\tau_M)$.
    \end{lem}
    
    \begin{proof}
        From Lemma~\ref{qnsep}, $Q_n(x, x^M)$ is separable (which follows from the separability of $q_{n, M}(x)$ due to the circular interlacing condition), so that all the factors of $Q_n$ in \eqref{factorQ} are distinct. Since the trace of an algebraic integer is the trace of its minimal polynomial, taking traces in \eqref{factorQ} gives \[\tr(Q_n(x, x^M)) =  \tr(\tau_M) + \sum_{i=1}^{L_n} m_i\tr(\Phi_{S_i}(x)),\] where $m_i \leq 2$ (as defined in \eqref{valuemulti}). Since the trace of the $m$-th cyclotomic polynomial is $\mu(m)$, for any $S_i$ we have $|\tr(\Phi_{S_i}(x))| \le 1$, and the claimed inequality follows.  
    \end{proof} 
    
    For $r > 1$, we have $S_r > S_1 \geq 2$. The next two lemmas give upper bounds on the size of $S_r$ and $L_n$.
    
    \begin{lem}\label{boundtr}
        For $n \geq 10^4$ and $1 \leq r \leq L_n$, $S_r < (6n\log n)^2$. 
    \end{lem}
    
    \begin{proof}
        For each $q \in \N \setminus \{1\}$ such that $\# \mathcal{C}_{n, q}^* = 1$, if $Q_n(\zeta_q, \zeta_q^t) = 0$, then $Q_n(\zeta_q^s, \zeta_q^{ts}) = 0$ for all $0< s < q$ and $\gcd(s, q) = 1$, since $\Phi_{q}(x) \mid Q_n(x, x^M)$ for all $M \equiv t \Mod q$. Therefore, \[\# \mathcal{C}_n^* = \sum_{\substack{q \in \N \\ \# \mathcal{C}_{n, q}^* = 1}} \deg \Phi_{q} = \sum_{\substack{q \in \N \\ \# \mathcal{C}_{n, q}^* = 1}} \varphi(q) = \sum_{r=1}^{L_n} \varphi(S_r) \leq \# \mathcal{C}_n < 22n^2 (\log n + \log \log n).\] Moreover, we have \begin{equation}\label{Trbd}
            \frac{S_{L_n}}{2\log\log S_{L_n} + 3} < \varphi(S_{L_n}) < \sum_{r=1}^{L_n} \varphi(S_r) < 22n^2 (\log n + \log \log n),
        \end{equation} where the inequality follows from~\cite[Theorem~15]{RS}, for $b > 2$, \[\varphi(b) > \frac{b}{e^{\gamma}\log\log b + \frac{3}{\log \log b}}.\] Here $\gamma = 0.5772\dots$ is the Euler--Mascheroni constant and $\e^{\gamma} < 2$.
        
          If $S_{L_n} \geq (6n\log n)^2$, then, for $n \geq 10^4$, LHS grows faster than RHS in \eqref{Trbd}, and, in fact, \[\frac{(6n\log n)^2}{2\log(2\log (6n\log n)) + 3} > 22n^2 (\log n + \log \log n).\] This follows from the fact that the function \[U(n):=\frac{22n^2 (\log n + \log \log n)}{\frac{(6n\log n)^2}{2\log(2\log (6n\log n)) + 3}} \, \text{is a decreasing function in $n$},\] and $U(10^4) < 1$. Therefore $S_{L_n} < (6n\log n)^2$ for $n \geq 10^4$, which completes the proof.
    \end{proof}
    
    \begin{lem}\label{boundL}
     For $n \geq 10^4$, the number of $S_r$, namely $L_n$, is bounded above by $7n\log n$.
    \end{lem}
    
    \begin{proof}
        Since $S_r \neq p_t$ for $1 \leq t \leq n$, $S_r > r$ for $r = 1, \dots, L_n$. Therefore, \begin{equation}\label{phitrbd}
            \sum_{r=1}^{L_n} \varphi(S_r) \geq \sum_{r=3}^{L_n} \frac{S_r}{2\log\log S_r + 3} \geq \sum_{r=3}^{L_n} \frac{r}{2\log\log r + 3} \geq \frac{\sum_{r=3}^{L_n} r}{2\log \log L_n + 3} = \frac{L_n^2 + L_n - 6}{2(2\log \log L_n + 3)},
        \end{equation} where the first inequality follows from~\cite[Theorem~15]{RS}, and the second inequality holds since the function $\frac{w}{2\log \log w + 3}$ is monotonically increasing in $w$ for $w \in [3, \infty)$.

        Therefore, combining \eqref{boundcylpt} and \eqref{phitrbd}, we have \[22n^2 (\log n + \log\log n) > \#\mathcal{C}_n \geq \sum_{r=1}^{L_n} \varphi(S_r) \geq  \frac{L_n^2 + L_n - 6}{2(2\log \log L_n + 3)}.\] If $L_n \geq 7n\log n$, then RHS grows faster than LHS since, for $n \geq 10^4$, a similar reason as in the proof of Lemma~\ref{boundtr} yields \[22n^2 (\log n + \log\log n) < \frac{(7n\log n)^2 + 7n\log n - 6}{2(2\log\log (7n\log n)) + 3},\] which is a contradiction. Therefore, $L_n < 7n\log n$ for $n \geq 10^4$.   
        \end{proof}
    
    Let $p_{d_n}$ denote the $d_n$-th prime which is the smallest prime such that \begin{equation}\label{boundpdn}
        S_{L_n} < (6n\log n)^2 \leq p_{d_n} < 2(6n\log n)^2.
    \end{equation}  
    For $1 \leq r \leq L_n$, the following lemma establishes an upper bound on the size of $\omega(S_r)$.
    
    \begin{lem}\label{sizeIr}
        For $n \geq 10^4$ and $1 \leq r \leq L_n$, the number of distinct prime factors of $S_r$, namely $\omega(S_r)$, is bounded above by $\frac{3}{2}\log n$. 
    \end{lem}
    
    \begin{proof}
        For every $S_r$, we have, for $\omega(S_r) \geq 3$, \begin{align*}
            \varphi(S_r) \geq \prod_{\substack{u \\ u \in \{1, \dots, d_n\}\\ p_u \mid S_r}} (p_{u} - 1) \geq& \, \prod_{t=1}^{\omega(S_r)}(p_t - 1) \\ \geq& \, \prod_{t=1}^{\omega(S_r) - 1} p_t  > \exp\left((\omega(S_r)-1)(\log (\omega(S_r)-1) + \log \log (\omega(S_r) - 1) -c_0\right),
        \end{align*}
        where the last inequality follows from \eqref{primorialbd}. Since $\sum_{r=1}^{L_n} \varphi(S_r) \leq \#\mathcal{C}_n < 22n^2(\log n + \log\log n)$, the above discussion yields \begin{align*}
            &\exp\left((\omega(S_r)-1)(\log (\omega(S_r)-1) + \log \log (\omega(S_r) - 1) -c_0\right) \\ <& \, \prod_{t=1}^{\omega(S_r) - 1} p_t \leq \varphi(S_r) \leq \#\mathcal{C}_n < 22n^2(\log n + \log\log n),
        \end{align*} and, after taking the logarithm on both sides, we note that if $\omega(S_r) \geq  \frac{3}{2}\log n$, then, because a similar reason as in the proof of Lemma~\ref{boundtr} implies \[\left(\frac{3}{2}\log n - 1\right)\left(\log\left(\frac{3}{2}\log n - 1\right) + \log\log\left(\frac{3}{2}\log n - 1\right) - c_0\right) > \log 22 + 2\log n + \log(\log n + \log\log n),\] for $n \geq 10^4$, we arrive at a contradiction. Therefore, $\omega(S_r) < \frac{3}{2}\log n$ for $n \geq 10^4$, and this concludes our proof.    
    \end{proof}
    
    \section{Upper bound on the degree of a Salem polynomial with negative trace}\label{PT1}
    
    Our main goal of this section is to prove the following theorem.
    
     \begin{thm}\label{finalthm}
            For sufficiently large odd positive integer $n$, there exists a Salem number with trace $-\lfloor\frac{n}{2}\rfloor + o(n)$ and degree at most $n(n+2)(\log n + \log\log n)$. 
         \end{thm}
    
    The error term in the statement can be made explicit. For one, it is $\leq \frac{n\log\log n}{\log n}$. We show that for $n \geq$ $10^4$, there is a Salem number with a trace in $\left[-\frac{3}{2}\lfloor\frac{n}{2}\rfloor, -\frac{1}{2}\lfloor\frac{n}{2}\rfloor\right]$ and a degree at most $n(n+2)(\log n + \log\log n)$, which reduces the previously known doubly exponential bound to a polynomial bound.

    In the rest of the section, we propose a construction of a suitable number, $k_n$ to be precise (defined below), of positive integers $M_j$, and count the number of distinct $S_r$ such that $\Phi_{S_r}(x)$ is a cyclotomic factor of $Q_n(x, x^{M_j})$. Then we want to argue that for at least one of such $M = M_j$, the absolute value of the trace of the corresponding $Q_n(x, x^{M_j})$, as well as the Salem number associated with it, grows with $n$. This is achieved in the following steps. \begin{itemize}
        \item \hyperref[sec:knmj]{\textbf{Constructing integers $M_j$}}: A set of $k_n = 8\lceil(\log n)^2\rceil$ consecutive integers $M_j$ is chosen such that $p_n < M_j < 2p_n$ (and therefore $\gcd(M_j, p_n) = 1$, see Section~\ref{congruencemj}). This construction (specifically, the coprimality part), along with the discussions in Section~\ref{preinter}, guarantees circular interlacing conditions such that the resulting polynomials $Q_n(x, x^{M_j})$ produce Salem numbers. (See Remark~\ref{choicesmallkn} for more details on the particular choice of $k_n$.)

        \item \hyperref[sec:cyclofactor]{\textbf{Estimation of the number of cyclotomic factors}}: According to Lemma~\ref{realtraces}, if $C_j$ denotes the number of cyclotomic polynomials $\Phi_{S_r}(x)$ that divide $Q_n(x, x^{M_j})$ (counting multiplicity, that is, $m_r$), then our ideal $M_j$ would be the one where $C_j$ is not large (that is, $o(n)$). To do that, for every $S_r$ we first bound the number of integers $M_j$ such that $Q_n(\zeta_{S_r}, \zeta_{S_r}^{M_j})=0$ (see Lemma~\ref{uppbdmrir} and \eqref{srinsteadtr}) and therefore the number of $M_j$, for which $\Phi_{S_r}(x)$ is a factor of $Q_n(x, x^{M_j})$. We use this bound in the next section to understand the sum of all $C_j$.

        \item \hyperref[sec:avgcj]{\textbf{Averaging $C_j$}}: By counting the sum of all $k_n$ many $C_j$ in two different ways leads to an upper bound of $\sum_{j=1}^{k_n} C_j$ (see Lemmas~\ref{boundsumcj} and~\ref{boundknphitr}).

        \item \hyperref[sec:pfthm]{\textbf{Conclusion of the bound, proof of Theorem~\ref{finalthm}}}: Comparing this upper bound with the total sum that would result if all $C_j$ were large $\left(\,\text{specifically}\, C_j \geq \frac{n\log\log n}{\log n}\right)$, we arrive at a contradiction for all large $n$ (more precisely for $n \geq 10^4$). The contradiction implies that there must be one $M_{j}$ where $C_{j} = o(n)$. For this specific $M_j$, the degree of the resulting polynomial $Q_n(x, x^{M_{j}})$ provides the new improved bound on the degree of the minimal polynomial of the Salem number with a trace lying between $-\frac{3}{2}\lfloor\frac{n}{2}\rfloor$ and $-\frac{1}{2}\lfloor\frac{n}{2}\rfloor$, for $n \geq 10^4$.
        
        \end{itemize}
    
    Next, we expand on these steps. 
    \bigskip
    
    \subsection{Constructing integers \texorpdfstring{$M_j$}{M\_j}}\label{sec:knmj}
    
    Note that $k_n = 8\lceil (\log n)^2 \rceil$ implies $k_n = O((\log n)^2)$. Moreover, \eqref{pmbound} and \eqref{primorialbd} yield, for $n \geq 10^4$, \begin{equation}\label{smallkn}
        k_n < 8(\log n)^2 + 8 < n\log n - 1 < p_n - 1.
    \end{equation} As mentioned above, we want to carefully choose consecutive integers $p_n < M_1 < \cdots < M_{k_n} < 2p_n$. The choice ensures that $\gcd(M_j, p_n) = 1$ for all $1 \leq j \leq k_n$. From \eqref{smallkn}, the choice also confirms that $M_i \not\equiv M_j \Mod {p_n}$ for all $i \neq j$. In the next section, we show that such a construction indeed exists.
    
    \bigskip
    
    \subsubsection{Choices of congruences}\label{congruencemj} Let $p_m$ be the $m$-th prime such that $ p_n < p_m < 2p_{n}$ and $m \geq n+1$.  For $n \geq 1$, the existence of such a prime follows from Bertrand's postulate~\cite[Corollary 3, eq. (3.9)]{RS}. Next, we choose $k_n$ consecutive integers $M_j$ such that $p_n < M_j < 2p_n$ for all $1 \leq j \leq k_n$, and, for some $1 \leq j_0 \leq k_n$, $M_{j_0} := p_m$. 
    
    Note that the choices of $M_j$ ensure that $\gcd(M_j, p_n) = 1$ and $M_j > p_n$, which is the necessary and sufficient condition for the ratio $\frac{x^{p_n + M_j} - 1}{(x^{p_n}-1)(x^{M_j} - 1)}$ to satisfy the circular interlacing condition (see \eqref{thefamily}). Furthermore, from \eqref{pmbound}, we have \[n\log n < p_n < M_j < 2p_n < 2n(\log n + \log\log n).\] Since $p_n < M_j < 2p_n$, \[p \mid M_j \Longrightarrow \left\{\begin{array}{cc}
       p \leq p_{n-1},  & \quad \text{when} \, \, \omega(M_j) \geq 2, \\
       p = p_m  & \quad \text{when} \, \, M_j = p_m.
    \end{array}\right.\] A similar argument as in Lemma~\ref{sizeIr} applied to $\omega(M_j)$ with the bounds in \eqref{pmbound} implies that, for $n \geq 10^4$, \[\omega_n^*(M_j) \leq \omega(M_j) \leq \log n,\] where $\omega_n^*(M)$ is as defined in \eqref{modifiedomega}.

    \subsection{Estimating the number of cyclotomic factors}\label{sec:cyclofactor}
    
    If $\tau_{M_j}$ is the Salem number such that $Q_n(\tau_{M_j}, \tau_{M_j}^{M_j}) = 0$, then, to obtain information about $\tr(\tau_{M_j})$, Lemma~\ref{realtraces} tells us that it is important to investigate \[ \#\{S: Q_n(\zeta_{S}, \zeta_{S}^{M_j}) = 0\}. \] To begin with, in the following discussions, we establish an upper bound for \[\sum_{j=1}^{k_{n}}  \#\{S: Q_n(\zeta_{S}, \zeta_{S}^{M_j}) = 0\}.\]    
    
    Our first step towards that goal is to upper bound the cardinality of the set \[\left\{M_j: 1 \leq j \leq k_n, \, M_j \equiv a \Mod {S},\, 0 \leq a < S\right\}\] for any $S$. In fact, we will consider the set, for integers $q > 1$ such that $\#\mathcal{C}_{n, q}^* = 1$, \[\mathcal{M}_q := \left\{M_j: 1 \leq j \leq k_n, \, M_j \equiv a \Mod {q},\, 0 \leq a < q\right\}.\]

    \begin{lem}\label{uppbdmrir}
        The cardinality of $\mathcal{M}_q$, is bounded above by $\frac{k_n}{q} + 1$.
    \end{lem} 
    
    \begin{proof}
       Since $M_j$ are consecutive integers, the number of $M_j$ that is congruent to $a \Mod q$ for a fixed residue $ 0 \leq a < q$ is at most\[\left\lceil\frac{k_n}{q}\right\rceil \leq \frac{k_n}{q} + 1,\] which completes the proof.
        \end{proof}
    
    This implies that for $q = S$, \begin{equation}\label{srinsteadtr}
       \# \mathcal{M}_{S} = \#\left\{M_j: 1 \leq j \leq k_n, \, M_j \equiv b \Mod {S}, 0 \leq b < S\right\} \leq \frac{k_n}{S} + 1.
    \end{equation}
    
    \medskip
    
    \subsection{Averaging \texorpdfstring{$C_j$}{C\_j}}\label{sec:avgcj}
    
        Recall that $C_j$ denotes the number of $S$ (counting multiplicity) such that $\Phi_{S}(x)$ divides $Q_n(x, x^{M_j})$, that is, \[C_j \leq  2\#\{S:  Q_n(\zeta_{S}, \zeta_{S}^{M_j}) = 0\}.\] By Lemma~\ref{realtraces}, it follows that the trace of the Salem number $\tau_{M_j}$ associated with $Q_n(x, x^{M_j})$ is bounded above by $-\lfloor \frac{n}{2}\rfloor +\omega_n^*(M_j) + C_j$ (where $\omega_n^*(M_j)$ is defined in \eqref{modifiedomega}), that is, \begin{equation}\label{tracebound}
            \tr(\tau_{M_j}) \leq \tr(Q_n(x, x^{M_j})) - \sum_{\substack{S \\ \Phi_{S}(x) \mid Q_n(x, x^{M_j})}} m_r\tr(\Phi_{S}(x)) \le  -\left\lfloor\frac{n}{2}\right\rfloor + \omega_n^*(M_j) + C_j,
        \end{equation} since $\tr(\Phi_{S}(x)) = \mu(S) \in \{-1,0,1\}$. From Lemma~\ref{realtraces}, we can further conclude that \begin{equation}\label{lowerbdspecific}
            \tr(\tau_{M_j}) \geq  -\left\lfloor\frac{n}{2}\right\rfloor + \omega_n^*(M_j) - C_j \geq -\left\lfloor\frac{n}{2}\right\rfloor - C_j.
        \end{equation}
        
        We argue that at least for some $j$ the absolute value of $\tr(\tau_{M_j})$ grows with $n$, that is, for example, $\tau_{M_j} \in \left[-\frac{3}{2}\lfloor\frac{n}{2}\rfloor, -\frac{1}{2}\lfloor\frac{n}{2}\rfloor\right)$ for some $j$ and sufficiently large $n$. Following \eqref{tracebound} and since $\omega(M_j) \leq \log n$, we therefore need to show that there exists $C_j$ that grows as $o(n)$. Once such $C_j$ is obtained, the corresponding degree of $Q_n(x, x^{M_j})$ is then an upper bound for the minimal polynomial of the Salem number $\tau_{M_j}$ with $|\tr(\tau_{M_j})| = O(n)$, and since $M_j < 2n(\log n + \log\log n)$, this improves the previously known bound in~\cite{MS}. The next lemmas are essential steps towards achieving our claimed improvement. 
    
    \medskip
    
    \subsubsection{Bounding sum of all $C_j$}
    
    To systematically count the cyclotomic factors across all chosen parameters $M_j$, we construct an incidence-like $L_n \times k_n$ matrix $\mathcal{D}_n$ with entries either $0$ or $1$ or $2$. The element of the matrix is defined as follows for $1 \leq r \leq L_n$ and $1 \leq j \leq k_n$,  
    \[\mathcal{D}_{n, rj} = \begin{cases}
        2 & \,  \text{when} \, \Phi_{S_r}^2(x) \parallel Q_n(x, x^{M_j}), \\
       1  & \, \text{when} \, \Phi_{S_r}(x) \parallel Q_n(x, x^{M_j}), \\
       0  & \, \text{when} \, \Phi_{S_r}(x) \nmid Q_n(x, x^{M_j}),
    \end{cases}\] where the bounds on multiplicity $m_r\in \{0,1,2\}$ follows strictly from the separability of $c_{n,M}$ in Lemma~\ref{qnsep}. Note that since, for all $r$, $\#\mathcal{C}_{n, S_r}^* = 1$, the matrix exhaustively captures, as well as computes, the cyclotomic factors $\Phi_{S_r}(x)$ of $Q_n(x, x^{M_j})$, counting multiplicity. Then we have the following result. 
    \begin{lem}\label{boundsumcj}
        We have \[\sum_{j=1}^{k_n} C_j \leq 2L_n + 2\sum_{r=1}^{L_n} \frac{k_n}{S_r}.\]
    \end{lem}
    
    \begin{proof}
        Note that \[\sum_{r=1}^{L_n} \mathcal{D}_{n, rj} = C_j, \quad \text{and} \quad \sum_{j=1}^{k_n} \mathcal{D}_{n, rj} \leq 2\#\mathcal{M}_{S_r}.\] Therefore, \[\sum_{j=1}^{k_n} C_j = \sum_{j=1}^{k_n}\sum_{r=1}^{L_n} \mathcal{D}_{n, rj} = \sum_{r=1}^{L_n}\sum_{j=1}^{k_n} \mathcal{D}_{n, rj} \leq 2\sum_{r=1}^{L_n} \#\mathcal{M}_{S_r} \leq 2\sum_{r=1}^{L_n} \left(\frac{k_n}{S_r} + 1\right) = 2L_n + 2\sum_{r=1}^{L_n} \frac{k_n}{S_r},\] where the last bound follows from Lemma~\ref{uppbdmrir}.
    \end{proof}
    
    \medskip
    
    \subsubsection{Bounding $\sum_{r=1}^{L_n} \frac{k_n}{S_r}$} In this section, we establish a bound of $\sum_{r=1}^{L_n} \frac{k_n}{S_r}$ for a given $k_n$, and use it in the next section to bound the growth of $\sum_{j=1}^{k_n} C_j$ in terms of $n$. We obtain such a bound in the next lemma.

    \begin{lem}\label{boundknphitr}
        For $n \geq 10^4$, \begin{equation}\label{bound ofkn/phisr}
            \sum_{r=1}^{L_n} \frac{k_n}{S_r}   < k_n\left(\log 7 + \log n + \log\log n +1\right).
        \end{equation}
    \end{lem}
    
    \begin{proof}
        By definition, the values $S_r$ are $L_n$ distinct integers greater than $1$. To maximise the sum of their reciprocals, it suffices to assume that they are the smallest such possible integers $2,3,\ldots,L_n+1$. Therefore, 
        \begin{align*}
            \sum_{r = 1}^{L_n} \frac{k_n}{S_r} \leq \sum^{L_n+1}_{i=2}\frac{k_n}{i} \leq k_n\int^{L_n+1}_{1}\frac{1}{x}dx&=k_n\log(L_n+1)\\
            &\le  k_n\log(7n\log n +1)\\ &=  k_n\left(\log(7n\log n)  + \log\left(1 + \frac{1}{7n\log n}\right)\right)\\
            &\le k_n(\log 7+\log n +\log\log n +1)
        \end{align*}
        where the penultimate inequality follows from Lemma~\ref{boundL}. 
    \end{proof}
    
    \medskip
    
    \subsection{Proof of Theorem~\ref{finalthm}}\label{sec:pfthm}
    
    Now we are ready to improve the upper bound of the degree of the minimal polynomial of the associated Salem numbers.
        
        \begin{proof}[Proof of Theorem~\ref{finalthm}]\label{pfthm}
            Collecting the bounds from previous sections and recalling that $k_n = 8\lceil (\log n)^2 \rceil$, we have \begin{align}
                \sum_{j=1}^{k_n} C_j \leq & \, 2L_n + 2\sum_{r=1}^{L_n} \frac{k_n}{S_r} \nonumber \\ <& \, 2L_n +  2k_n\left(\log 7 + \log n + \log\log n +1\right) \nonumber \\ <& \, 14n\log n +  16((\log n)^2 +  1)\left(\log 7 + \log n + \log\log n+1\right) \label{ultmbd}
            \end{align} where the final inequality follows by substituting $L_n<7n\log n$ (Lemma~\ref{boundL}) and $k_n<8((\log n)^2+1)$. 
    
       On the other hand, if $C_j \geq \frac{n \log\log n}{\log n}$ for all $j$, then, since $k_n = 8\lceil (\log n)^2 \rceil$, we obtain \[8n\log n\log\log n \leq k_n\frac{n\log\log n}{\log n} \leq \sum_{j=1}^{k_{n}} C_j.\] Since, for $n \geq 10^4$, a similar reason as in the proof of Lemma~\ref{boundtr} gives \begin{align}
           \frac{14n\log n +  16((\log n)^2 +  1)\left(\log 7 + \log n + \log\log n+1\right)}{8n\log n\log\log n}  < \, 1, \label{choiceofkn}
       \end{align} we arrive at a contradiction. Therefore, for $n \geq 10^4$, there exists at least one $j_0 \in \{1, \dots, k_n\}$ such that $C_{j_0} < \frac{n\log\log n}{\log n}$, that is $C_{j_0} = o(n)$. This, combined with Lemma~\ref{realtraces} and the inequality in \eqref{lowerbdspecific}, implies that, for the corresponding $j_0$, \[\tr(\tau_{M_{j_0}}) \in \left[-\left\lfloor\frac{n}{2}\right\rfloor  - C_{j_0}, -\left\lfloor\frac{n}{2}\right\rfloor + \omega_n^*(M_{j_0}) + C_{j_0}\right],\] where $\tau_{M_{j_0}}$ is the Salem number which is a root of $Q_{n}(x, x^{M_{j_0}})$, and, since $C_{j_0} = o(n)$ (and $\omega_n^*(M_{j_0}) \leq \log n= o(n)$), we obtain that $\tr(\tau_{M_{j_0}}) = -\left\lfloor\frac{n}{2}\right\rfloor + o(n)$, as we proposed to show.
        
        Further, from the construction of $M_j$, we have $p_n \nmid M_j$ and $M_j > p_n$ for all $1 \leq j \leq k_n$. Therefore, all such $M_j$ and corresponding $Q_n(x, x^{M_j})$ satisfy the interlacing conditions and produce a Salem number with trace bounded above by $-\lfloor \frac{n}{2} \rfloor + \omega_n^*(M_j) + C_j$. Moreover, since \[M_j < 2p_n < 2n(\log n + \log\log n)\] for $j$, we have \[ \deg Q_n(x, x^{M_j}) = M_j - \sum_{p \mid M_j}' \deg \Phi_p (x) + 2- n + \sum_{t=1}^n p_t = M_j - \sum_{p \mid M_j}'(p-1) + 2- n + \sum_{t=1}^n p_t,\] where the sum $\sum_{p \mid M_j}' (p-1)$ is $0$ when $M_j$ is a prime, otherwise it is $\sum_{p \mid M_j} (p-1)$, and, since $\sum_{p \mid M_j} (p-1) \leq \prod_{p \mid M_j} p \leq M_j < 2p_n$, we derive, using \eqref{pmbound} and \eqref{sumpmbound}, for $ n \geq 6$, \begin{equation}\label{degreebduplow}
            4 < 2- n + \sum_{t=1}^n p_t < \deg Q_n(x, x^{M_j}) <  2p_n +2 - n + \sum_{t=1}^n p_t < n(n+2)(\log n + \log\log n),
        \end{equation}  which concludes the proof.
        \end{proof}

        \begin{rem}\label{choicesmallkn}
            The reason behind the particular choice of $k_n$ is related to obtaining a contradiction as in \eqref{choiceofkn}. In particular, first note that the term $7n\log n$ is the dominant term in \eqref{ultmbd}, which is the upper bound of $L_n$, a term independent of $k_n$. To obtain a contradiction under the assumption $C_j \geq \frac{n\log \log n}{\log n}$, we need to choose $k_n$ so that the denominator of the LHS in \eqref{choiceofkn} grows faster than $n\log n$ as $n$ grows. Although our choice satisfies that, it is not optimised, as any choice of $k_n$ of the size $O\left(\frac{\log n}{(\log \log n)^{1-\epsilon}}\right)$ should work ($ \epsilon > 0$), but at the cost of the lower bound ($n_0$) of such $n$ where our theorem should hold (that is, for all $n \geq n_0$). Indeed, for example, in the cases with exponent $\epsilon < 1$, the lower bound is much larger than $10^4$. Our choice of $k_n$ is therefore based on obtaining a smaller value of $n_0$ and also satisfying \eqref{smallkn} and \eqref{choiceofkn}. 
        \end{rem}

    \medskip
    
    \section{Proof of Theorem~\ref{thm:uplowbd}}\label{sec:mainproof}
    \subsection{Bounds on \texorpdfstring{$D(T)$}{D(T)}}\label{sec:dict}
    
    Recall that for an algebraic number $\beta$, $\tr(\beta)$ is defined as the trace of its minimal polynomial over $\Z$. Throughout this section, $\tau$ is a Salem number of degree $2d$ and $\tr(\tau)$ is a negative integer. Recall that for $T \in \N$,
    	\[
    	D(T)=\min\{\deg\tau:\ \tr(\tau)\le -T\}.
    	\] Then we have the following lemma.
    	
    	\begin{lem}\label{lem:dict}
    		For a Salem number $\tau$ of degree $2d\ge 4$ with $\tr(\tau) = -T$ for some positive integer $T$, put
    		\[
    		\alpha:=\tau+\tau^{-1}+2 .
    		\]
    		Then $\alpha$ is a totally positive algebraic integer of degree exactly $d$, $ \alpha > 4$, and all the other conjugates are in the interval $(0, 4)$. Moreover,
    		\[
    		\tr(\alpha)=2d-T,\qquad \alpha <\tr(\alpha)<2d .
    		\]
    		Conversely $\deg\tau=2\deg\alpha$.
    	\end{lem}
    	
    	\begin{proof} From \cite[Proposition~3 (i)]{SSurvey} (see also \cite[p.~169]{Salem}), it follows that $\tau + \tau^{-1} = \alpha - 2 > 2$ is an irrational algebraic integer of degree $d$ with all the other conjugates in $(-2, 2)$. Adding $2$ gives the stated ranges, so $\alpha$ is totally positive and of degree exactly $d$.
    		
    		 Finally, since all the other conjugates of $\tau$ lie on the unit circle except $\tau^{-1}$, we have 
    		\[
    		\tr(\alpha)=\tr(\tau + \tau^{-1}) + 2d = \tr(\tau)+2d=2d-T,
    		\]
    		and $\alpha <\tr(\alpha)$ because $\alpha$ is totally positive of degree $d\ge2$. Moreover, $\tr(\alpha)=2d-T<2d$.
    	\end{proof}

        The next proposition provides an asymptotic upper bound for $D(T) $.
    
        \begin{prop}\label{prop:upper}
            For all sufficiently large $T$, there is a Salem number
    		$\tau$ with $\tr\tau\le-T$ and
    		\[
    		\deg\tau\ \le\ 2\Big(1+o(1)\Big)T^{2}\log T .
    		\]
    		Therefore, \begin{equation}\label{upperbdDT}
    		    D(T)\le(2+o(1))\,T^{2}\log T.
    		\end{equation} 
    	\end{prop}
    	
    	\begin{proof}
    		Let $n$ be odd, $n \ge 10^4$, and let $M_{j_0}$, $Q_n(x,x^{M_{j_0}})$ and $\tau_{M_{j_0}}$ be as in the
    		proof of Theorem~\ref{finalthm}. Then
    		\[
    		\tr(\tau_{M_{j_0}})\ \le\ -\Big\lfloor\frac n2\Big\rfloor+\omega_n^{*}(M_{j_0})+C_{j_0}
    		\ \le\ -\Big\lfloor\frac n2\Big\rfloor+\log n+\frac{n\log\log n}{\log n}.
    		\]
    		Hence $\tr(\tau_{M_{j_0}})\le -T$ as soon as
    		$\lfloor n/2\rfloor-\log n-n\log\log n/\log n\ge T$, which is true for
    		$n=2T\big(1+O(\log\log T/\log T)\big)$ where $n$ is odd.
    		For the degree, the estimate in \eqref{degreebduplow} is
    		\[
    		\deg Q_n(x,x^{M_{j_0}})\ \le\ 2p_n+2-n+\sum_{t=1}^{n}p_t .
    		\]
    		Instead of \eqref{sumpmbound}, $\sum_{t\le n}p_t<n^2(\log n+\log\log n)$, we use the
    		asymptotic consequence of the prime number theorem \cite[Theorem~1.4]{Axler}
    		\[
    		\sum_{t\le n}p_t=\Big(\tfrac12+o(1)\Big)n^{2}\log n ,
    		\]
    		together with $p_n=O(n\log n)$. With $n=(2+o(1))T$ this gives
    		\[
    		\deg Q_n(x,x^{M_{j_0}}) \leq \Big(\tfrac12+o(1)\Big)n^2\log n=\big(2+o(1)\big)T^{2}\log T .
    		\]
    		The minimal polynomial of $\tau_{M_{j_0}}$ divides $Q_n(x,x^{M_{j_0}})$, so its degree is at most this.

            Since $D(T)=\min\{\deg\tau:\ \tr(\tau)\le -T\}$, \eqref{upperbdDT} follows from above.
    	\end{proof}
    
       For the lower bound on $D(T)$, first note that, by Lemma~\ref{lem:dict}, $\alpha=\tau+\tau^{-1}+2$ is totally positive of degree $d=\tfrac12\deg\tau$ with $\tr(\alpha)=2d+\tr(\tau)$. If $\tr\alpha\ge\lambda d$ then $|\tr(\tau)|\le(2-\lambda)d$, i.e.
    		$\deg\tau=2d\ge \tfrac{2}{2-\lambda}|\tr(\tau)|$. From \cite[Corollary~1.9]{OSS24}, we take $\lambda=1.80203$ to obtain \begin{equation}\label{lowerbdDT}
    			    \deg\tau\;\ge\;\frac{2}{2-1.80203}\,|\tr(\tau)|\;>\;10.1025\,|\tr(\tau)| .
    			\end{equation} Now using this and noting that there are only finitely many Salem numbers $\tau$ with negative trace for which \eqref{lowerbdDT} does not hold, we consider $T_0$ to be an integer larger than the absolute values of the traces of these exceptions, then, for all $T \geq T_0$, there exists at least one Salem number $\tau$ such that $\tr(\tau) = -T$ and satisfies \eqref{lowerbdDT} (the existence follows from \cite[Theorem~1.1]{MS}). The above discussion and the definition of $D(T)$ then yield 
                \begin{equation}\label{ineq:linlower}
                    D(T) > 10.1025 T,
                \end{equation} giving an effective lower bound of $D(T)$. 
     
        In what follows, we conclude that for sufficiently large $T$, $D(T)$ is in fact at least $O(T^{2-\epsilon})$ for arbitrarily small $\epsilon$. 
    
        \begin{prop}\label{prop:main}
    Let $\tau$ be a Salem number of degree $2d$ with $\tr(\tau)=-T$, $T\ge1$. Then
    \[
    T^{2}\ \le\ 8\,d\log(2d)+11d,
    \]
    and consequently $D(T)\ \ge\ \dfrac{2T^{2}}{16\log(2T) + 11}$ for every $T\ge1$.
    \end{prop}
    
    \begin{coro}\label{cor:nolinear}
    For every $\lambda>0$ there are only finitely many Salem numbers with
    $\tr(\tau)\le-\lambda\deg(\tau)$, that is \[D(T)/T\to\infty.\] 
    \end{coro}
    
    \begin{proof}[Proof of Corollary~\ref{cor:nolinear} using Proposition~\ref{prop:main}]
    If $\tr(\tau) = -T$, $2\lambda d\le T$ and $T^{2}\le8d\log(2d) +11d$ then $4\lambda^{2}d\le8\log(2d) +11$,
    which holds for only finitely many $d$. Since all the conjugates of a Salem number $\tau$ are bounded above by the degree of its minimal polynomial (by Lemma~\ref{lem:dict}) and the coefficients of its minimal polynomial are symmetric functions of the conjugates, the coefficients of the minimal polynomial are bounded when the degree is fixed. Therefore, for each fixed degree, there are finitely many Salem numbers $\tau$ of bounded trace. 
    
    Since $D(T) \geq \dfrac{2T^{2}}{16\log(2T) +11}$, we have $\dfrac{D(T)}{T} \rightarrow \infty$ as $T \rightarrow \infty$.
    \end{proof}
    
    \begin{rem}\label{rem:attribution}
    The qualitative statement of Corollary~\ref{cor:nolinear} is not new, as it is a special case of
    Pritsker's equidistribution theorem \cite[Corollary~2.6]{Pri11} applied to the minimal polynomial of $\alpha=\tau+\tau^{-1}+2$ with $E=[0,4]$.
    In~\cite[Corollaries~3.5 and~3.6]{Pri11}, Pritsker shows that if all roots of $P_n$ lie in a bounded interval of capacity $1$ (for example $[0,4]$ or $[-2,2]$) and the leading coefficients are bounded, then the arithmetic mean of the roots differs from the midpoint of the interval by $O(\sqrt{\log n/n})$; our setting differs in that one root, $\alpha=\tau+\tau^{-1}+2$, lies outside $[0,4]$.
    \end{rem}
    
    Next, we have a lemma that is essential to prove Proposition~\ref{prop:main}.
    
    \begin{lem}\label{lem:key}
    Let $w_{1},\dots,w_{N}$ be distinct points on the unit circle, $N\ge2$, and denote
    $q:=\sum_{k=1}^{N}w_{k}$. Then, for every $r\in(0,1)$,
    \[
    r\,|q|^{2}\ \le\ -\sum_{1\le k\ne l\le N}\log|w_{k}-w_{l}|
    \;+\;\frac{N(N-1)}{2}\log\frac1r\;+\;N\log\frac{1}{1-r}\,.
    \]
    \end{lem}
    
    \begin{proof}
    Let $q_{m}$ denote $\sum_{k}w_{k}^{m}$ for $m \in \N$. For $u \in \C$ with $|u|=1$ and $0<r<1$, one has
    \[-\log|1-ru|=-\frac{1}{2}\log|1-ru|^2= - \frac{1}{2}\left(\log(1-ru) + \log(1-r\overline{u})\right) = \sum_{m\ge1}\frac{r^{m}}{m}\operatorname{Re}(u^{m}),\] an absolutely convergent series.
    Summing over all $N^{2}$ ordered pairs $(k, \ell)$ and interchanging the sums (which is possible since
    $N^{2}\sum_{m \geq 1} \frac{r^{m}}{m}<\infty$), we have
    \begin{equation}\label{eq:psd}
    \sum_{k,\ell=1}^{N}\Bigl(-\log\bigl|1-r\,w_{k}\overline{w_{l}}\bigr|\Bigr)
    =\sum_{m\ge1}\frac{r^{m}}{m}\sum_{k,\ell = 1}^N\re\bigl(w_{k}^{m}\overline{w_{\ell}^{m}}\bigr)
    =\sum_{m\ge1}\frac{r^{m}}{m}\,|q_{m}|^{2}\ \ge\ r\,|q|^{2},
    \end{equation}
    where the last equality follows from the fact that \[|q_m|^2 = q_m\overline{q_m} = \re(q_m\overline{q_m}) = \re\left(\sum_{k, \ell = 1}^Nw_k^m \overline{w_\ell^m}\right) = \sum_{k, \ell = 1}^N\re(w_k^m \overline{w_\ell^m}),\] and, since all terms of the last series are non-negative, the last inequality holds by just considering the term with $m = 1$.
    
    We now bound the left side of~\eqref{eq:psd} from above. For each $k=\ell$ the term is
    $-\log |1-r| = -\log(1-r)$ (as $0 < r < 1$), giving $N\log\frac{1}{1-r}$ in total. For $k\neq \ell$, writing $u=w_{k}\overline{w_{\ell}}$ then yields
    \[
    |1-ru|^{2}=(1-r)^{2}+2r\left(1-\operatorname{Re}u\right) \geq 2r\left(1-\operatorname{Re}u\right) = r\,|1-u|^{2}=r\,|w_{k}-w_{\ell}|^{2},
    \]
    where we use $|w_{\ell}|=1$ in the last equality. Hence, taking the logarithm on both sides, we obtain
    $-\log|1-ru|\le\frac12\log\frac1r-\log|w_{k}-w_{\ell}|$, and summing over the $N(N-1)$ ordered pairs $k\ne \ell$ proves the claim.
    \end{proof}
    
    Now we are ready to prove our proposition. 
    
    \begin{proof}[Proof of Proposition~\ref{prop:main}]
        For a Salem number $\tau$ of degree $2d$ with $\tr(\tau)=-T$ (where $T \in \N$), by \cite{S-1}, a
    Salem number of negative trace has degree at least $8$, so $d\geq 4$. Let $P_{\tau}\in\Z[x]$ be its minimal polynomial, with roots
    \[
    \tau,\ \tau^{-1},\ w_{1},\dots,w_{N},\quad N:=2d-2 \geq 6,
    \]
    all distinct with $|w_k| = 1$ for all $1 \leq k \leq N$. Note that $\overline{w_k} = w_k^{-1} \in \{w_1, \dots, w_N\}$.  Write $q:=\sum_{k}w_{k}$, then the above implies that $q$ is a real
    number. Moreover \begin{equation}\label{lowerbdp}
        -T = \tr(\tau) = \tau + \tau^{-1} + q \Longrightarrow |q| = T + \tau + \tau^{-1} > T. 
    \end{equation}
    
    From Lemma~\ref{lem:dict}, we have \begin{equation}\label{eq:tausize}
        \tau < \tau + \tau^{-1} + 2 \leq \tr(\tau + \tau^{-1}+2) \leq 2d.
    \end{equation}
    
    As $P_{\tau}$ is irreducible, $\disc(P_{\tau})$, defined as \[\disc (P_{\tau}) = (\tau - \tau^{-1})^2\prod_{k = 1}^N(\tau - w_k)^2\prod_{k = 1}^N(\tau^{-1} - w_k)^2 \prod_{1 \leq k < \ell \leq N} (w_k - w_\ell)^2\] is a non-zero rational
    integer, so $|\disc(P_{\tau})|\geq 1$ and
    \[0 \leq \log|\disc(P_{\tau})|= 2\log|\tau-\tau^{-1}| + 2\sum_{k = 1}^N\log|\tau-w_{k}|+2\sum_{k = 1}^N\log|\tau^{-1}-w_{k}|
    + 2\sum_{1 \leq k < \ell \leq N}\log|w_{k}-w_{\ell}|.\]
    Since $|\tau-w_{k}|\le\tau+1<2\tau$,
    $|\tau^{-1}-w_{k}|\le1+\tau^{-1}<2$ and $|\tau-\tau^{-1}|<\tau$, we have
    \begin{equation}\label{eq:budget}
    -\sum_{\substack{1 \leq k, \ell \leq N \\ k \neq \ell}}\log|w_{k}-w_{\ell}| = -2\sum_{1 \leq k < \ell \leq N}\log|w_{k}-w_{\ell}| \le\ 2N\log(4\tau) +2\log\tau\ \le\ 2N\log(8d) +2\log(2d),
    \end{equation}
    where the last inequality follows from \eqref{eq:tausize}.
    
    Now using \eqref{eq:budget} \[\frac{N(N-1)}{2}\log\frac{N}{N-1} = \frac{N(N-1)}{2}\log \left(1 + \frac{1}{N-1}\right) \leq \frac{N(N-1)}{2}\cdot\frac{1}{N-1} = \frac{N}{2},\]  an application of Lemma~\ref{lem:key} to $w_{1},\dots,w_{N}$ with $r=1-\frac1N$ yields
    \[r\,|q|^{2}\ \le\ 2N\log(8d) +2\log(2d) +\frac N2+N\log N.\]
    Since $N=2d-2<2d$, we have \[N\log N<2d\log(2d), \quad 2N\log(8d) <4d\log(2d) +4d\log 4, \quad
    \frac N2<d, \; \text{and} \; 2\log(2d) \leq 2d.\] Hence
    \[
    r\,|q|^{2}\ <\ 6d\log(2d) +\bigl(4\log4+3\bigr)d\ <\ 6d\log(2d) +9d .
    \]
    Finally $\frac{1}{r}=\frac{N}{N-1}\leq \frac{6}{5}$ as $N\geq 6$, and this implies that
    \begin{equation}\label{ultbd}
        T^{2}<|q|^{2}\ <\ \tfrac65\bigl(6d\log(2d) +9d\bigr)\ =\ 7.2\,d\log(2d)+10.8\,d
    \ \le\ 8d\log(2d) +11d,
    \end{equation}
    where the first inequality follows from \eqref{lowerbdp}.
    
    For every Salem number $\tau$ of degree $2d \geq 8$ and $\tr(\tau)= -T'$ with some integer $T' \geq T$, the above discussion implies that
    \[T^{2} \leq T'^2 \leq 8d\log(2d) +11d.\] If $d \le T^{2}$ then $\log(2d)\leq \log(2T^{2})\leq 2\log(2T)$, and therefore, \eqref{ultbd} becomes 
    $T^{2}\le d(16\log(2T) +11)$, which implies that \[2d\geq \frac{2T^{2}}{(16\log(2T) +11)}.\] Since $D(T) = \min\{\deg\tau:\ \tr(\tau)\le -T\}$ and $T^2 > \frac{2T^{2}}{16\log(2T) +11}$, combining the above with the case when $d > T^2$, we derive that the minimum of all such $\deg \tau$ with $\tr(\tau) \leq -T$ is at least $\frac{2T^{2}}{16\log(2T) +11}$, that is \[D(T) \geq \frac{2T^{2}}{16\log(2T) + 11},\] which completes the proof.
    \end{proof}
    
     \begin{rem}\label{rem:sharp}
    Neither constant is optimal. Optimising $r$ and replacing the crude bounds $|\tau-w_{k}|<2\tau$ and $\tau<2d$ gives $T\le(1+o(1))\sqrt{6\,d\log d}$, hence $D(T)\ge(\tfrac13+o(1))\,T^{2}/\log T$. Numerically, Proposition~\ref{prop:main} improves on the linear bound~\eqref{ineq:linlower} only for $T \geq 700$. Lemma~\ref{lem:key} is sharp up to $O(N)$. Taking $q=0$ and $r=1-\tfrac1N$ gives $\sum_{k\ne l}\log|w_k-w_l|\le N\log N+\tfrac N2$, whereas the true maximum $N\log N$ is attained by the $N$-th roots of unity. The use of $\disc\neq 0$ in this style goes back to Schur~\cite{Sch18} and Siegel~\cite{Siegel}, and appears in the potential-energy form in Cherubini--Yatsyna~\cite{CY22} and Dubickas~\cite{Dub23}.
    \end{rem}
    
    \subsection{Proof of Theorem~\ref{thm:uplowbd}}
    	We now have everything needed to prove Theorem~\ref{thm:uplowbd}.
    
        \begin{proof}[Proof of Theorem~\ref{thm:uplowbd}]
            Collecting the results of Proposition~\ref{prop:upper} and Proposition~\ref{prop:main}, we conclude that, for sufficiently large $T$, \[\frac{2T^2}{16\log (2T) + 11} \leq D(T) \leq 2(1 + o(1))T^2\log T,\] which completes the proof.
        \end{proof}
    
    \medskip
    
    \section{Salem numbers with a prescribed trace of sufficiently large degrees}\label{theorem103}
    
    In this section, we prove Theorem~\ref{salemlargedegree}, which states that, for an odd positive integer $n$, there are Salem numbers of degree $2d$ and trace $-\lfloor\frac{n}{2}\rfloor$ for all sufficiently large $d$. 
    
    Before proceeding with the proof, we briefly recall some notations and framework from~\cite{MS} that we need. In order to obtain Salem numbers of trace $-\lfloor\frac{n}{2}\rfloor$, McKee and Smyth consider the following pairs of circular interlacing polynomials: $A(z) = z-1$ and $B(z) = z+1$. The discussion preceding~\cite[Proposition 3.2]{MS} ensures that the pair $A(z^m)$ and $B(z^m)$ also satisfy circular interlacing, for all $m \geq 1$. From Section~\ref{preinter}, we conclude that any finite sum of fractions $\frac{B(z^m)}{A(z^m)}$ (for different values of $m$) produces a pair of circular interlacing polynomials. In this regard, they consider the following polynomial in $n+2$ variables~\cite[Section 5]{MS}, where $(x, x_0, x_1, \dots, x_n) \in \mathbb{G}_m^{n+2}$ with $\mathbb{G}_m$ denoting the multiplicative group of $\C$, \begin{equation}\label{MSpoly}
        h(x, x_0, x_1, \dots, x_n) = 2(x^2 -1)\prod_{i=0}^n(x_i - 1) - x\sum_{j=0}^n(x_j+1)\prod_{\substack{i= 0 \\ i \neq j}}^n (x_i - 1),
    \end{equation} and, as mentioned in their article, the reason for looking at this polynomial is to be able to apply the identity \[\frac{h(x, x_0, x_1, \dots, x_n)}{2x\prod_{i=0}^n (x_i - 1)} = \frac{x^2 - 1}{x} - \frac{1}{2}\sum_{i=0}^n \frac{x_i+1}{x_i-1},\] which utilises the circular interlacing condition mentioned above when replacing $x_i$ with suitable powers of $x$. Following their argument in~\cite[Lemma 5.1, Lemma 5.2, Lemma 6.1]{MS}, we have positive integers $\ell_0, \dots, \ell_n$ such that \[h(x, x^{\ell_0}, \dots, x^{\ell_n}) = 2(x-1)^n P_{\tau, n}(x) \in \Z[x],\] where $P_{\tau, n}(x)$ is monic irreducible and a minimal polynomial of a Salem number $\tau_n$ with trace $-\lfloor\frac{n}{2}\rfloor$. In particular, their choice of $\ell_i$ are \begin{equation}\label{killerexponent}
        \ell_0 = \left(\prod_{\substack{p \leq 3\cdot 2^{n+1} \\ p \, \text{prime}}} p\right)\lcm\left(1, 2, \dots, \left\lceil(n+5)^{n+2}(n+2)^{\frac{n+2}{2}}\right\rceil\right),
    \end{equation} and $\ell_1, \dots, \ell_n$ are the smallest $n$ primes coprime to $\ell_0$ (the \textit{killer exponent}). In fact, the proof of Lemma 6.1 in~\cite{MS}, suggests that, for $i =1, \dots, n$, any (distinct) choice of $\ell_i$ which is a prime and $\gcd(\ell_i, \ell_0) = 1$ will produce a Salem number of the above prescribed trace, and the cyclotomic factor of the corresponding $h(x, x^{\ell_0}, \dots, x^{\ell_n})$ is $(x-1)^{n}$. For an odd positive integer $n$ and a given choice of $(\ell_0, \ell_1, \dots, \ell_n)$, the degree of the corresponding minimal polynomial of the Salem number is \[\deg \frac{h(x, x^{\ell_0}, \dots, x^{\ell_n})}{2(x-1)^n} = \deg P_{\tau, n}(x) = 2 + \sum_{i=0}^n \ell_i - n,\] an even integer. This brings us to our following proposition.
    
    \begin{prop}\label{alllarged}
        Let $\ell_0$ be as given in \eqref{killerexponent}. Let $K_n:= 2^n\ell_0$. Then, for $n \geq 10^4$, every odd integer greater than $2K_n^2$ can be represented as the sum of $n$ distinct primes which are coprime to $\ell_0$.
    \end{prop}
    
    Before giving a proof to Proposition~\ref{alllarged}, we show how to derive our theorem with its help. 
    
    \begin{proof}[Proof of Theorem~\ref{salemlargedegree} using Proposition~\ref{alllarged}]
        From Proposition~\ref{alllarged}, we have, for every $k \geq K_n^2$, there exist distinct primes $\ell_i$ (for $i =1, \dots, n$) which are coprime to $\ell_0$ such that \[\ell_1 + \cdots + \ell_n = 2k + 1.\] Since $\ell_0$ is even by definition in \eqref{killerexponent} and $n$ is odd, adding $\ell_0 + 2 - n$ both sides yields \[2 + \sum_{i=0}^n \ell_i - n = 2k + 1 + \ell_0 + 2 - n = 2d,\] where \begin{equation*}
            d \geq \frac{2K_n^2 + \ell_0 + 3 - n}{2} = \frac{1}{2}\left((2^{2n+1}\ell_0 + 1)\ell_0 + 3 - n\right).
        \end{equation*} Noting that the sum on the LHS is the degree of a minimal polynomial of a Salem number of trace $-\lfloor\frac{n}{2}\rfloor$, we conclude that for suitable choices of $(\ell_1, \dots, \ell_n)$, we obtain Salem numbers of degree $2d$ with trace $-\lfloor\frac{n}{2}\rfloor$ for all $n \geq 10^4$ and all $d \geq d_0(n)$, where \begin{equation}\label{don}
            d_0(n) := \frac{1}{2}\left((2^{2n+1}\ell_0 + 1)\ell_0 + 3 - n\right),
        \end{equation} and this completes the proof.  
    \end{proof}
    
    It only remains to deduce Proposition~\ref{alllarged}. 
    
    \begin{proof}[Proof of Proposition~\ref{alllarged}]
        Since $\ell_0$ is even, we consider first $n-3$ primes larger than $\ell_0$, and denote them by $\ell_1 < \cdots <\ell_{n-3}$. From Bertrand's postulate~\cite[Corollary 3, eq.~(3.9)]{RS}, we have $\ell_i \leq 2^i \ell_0$ for $i = 1, \dots, n-3$. Moreover, \[\sum_{i=1}^{n-3} \ell_i \leq \ell_0\sum_{i=1}^{n-3} 2^i = (2^{n-2} - 2)\ell_0 < K_n < K_n^2.\] Therefore, it only remains to show that, for every $k \geq K_n^2$, there are distinct primes $\ell_{n-2}$, $\ell_{n-1}$ and $\ell_n$, which are coprime to $\ell_0$ and bigger than $\ell_{n-3}$, such that \[\ell_{n-2} + \ell_{n-1} + \ell_n = 2\left(k -\frac{\sum_{i=1}^{n-3} \ell_i}{2}\right) + 1.\] Using the bound of $k$, we note that \[2\left(k -\frac{\sum_{i=1}^{n-3} \ell_i}{2}\right) + 1 > 2\left(k -(2^{n-3} - 1)\frac{\ell_0}{2}\right) + 1 > K_n^2.\]
        
        On the other hand, due to Vinogradov~\cite{Vino} (see~\cite[p. 146]{Daven} and~\cite{LLM} for our formulation), we have that all sufficiently large odd numbers can be represented as a sum of three primes. Combining the result in~\cite[Theorem~1]{LW} and Vinogradov's theorem, we obtain that, for every odd integer $M \geq e^{3100}$, we have \begin{equation}\label{repn3primes}
            r(M) \geq \frac{1}{4}\mathcal{G}(M)\frac{M^2}{(\log M)^3} > \frac{M^2}{4(\log M)^3},
        \end{equation} where $1 < \mathcal{G}(M) < 3$, and $r(M)$ is the number of ways any odd integer $M$ can be written as a sum of three primes.\footnote{In 2015, in a series of papers, Helfgott~\cite{Helf1} showed that we can take $M \geq 10^{27}$.} In the next three lemmas (Lemma~\ref{sameprimes}, Lemma~\ref{smallprimes} and Lemma~\ref{thebound}) we show that, for $n \geq 10^4$, the number of representations of an odd integer $M$ larger than $K_n^2$ as a sum of three primes, where either at least one of the primes is less than $\ell_{n-3}$ or at least two of them are equal, is strictly smaller than $\frac{M^2}{4(\log M)^3} - 1$. This ensures that there exist three distinct primes larger than $\ell_{n-3}$ such that $M$ is the sum of those three primes, which completes our proof.
    \end{proof}
    
    Next, we prove our claimed results in the argument above using the following lemmas.
    
    \begin{lem}\label{sameprimes}
        For an odd integer $M$, the number of representations of $M$ as a sum of three primes where at least two of them are equal is bounded above by $3\pi\left(\frac{M}{2}\right)$, where $\pi(\alpha)$ denotes the number of primes up to $\alpha \in \R_{> 0}$.
    \end{lem}
    
    \begin{proof}
        We want to count the number of prime tuples $(q_1, q_2, q_3)$ such that at least two of them are equal, and $q_1 + q_2 +q_3 = M$. Without loss of generality, assuming $q_1 = q_2$, we have $2q_1 + q_3 = M$. Noting that for a given $M$ and a given choice of $q_1$, the choice of $q_3$ is already determined. Therefore, the number of such representations is determined by the choice of $q_1$. Since $2q_1 < M$, the number of choices for $q_1$ (for a given odd integer $M$) is at most $\pi\left(\frac{M}{2}\right)$, and because of the symmetry in our choice of $q_1$, $ q_2$ and $q_3$, we obtain our result. 
    \end{proof}
    
    \begin{lem}\label{smallprimes}
        For an odd integer $M> 2 \cdot 10^{24}$, the number of ways $M$ can be written as a sum of three primes with at least one prime less than $K_n$ is bounded above by $6K_n\frac{M}{\log M}$.
    \end{lem}
    
    \begin{proof}
        If there exist three primes $q_1$, $q_2$, and $q_3$ such that $q_1 \leq K_n$ and $M = q_1 + q_2 + q_3$, then the number of choices of the tuple of primes $(q_2, q_3)$ such that $q_2 \geq q_3$ and $q_2 + q_3 = M - q_1$ is bounded above by, from~\cite[Lemma 3]{DGNP}, \[0.961 \frac{\frac{M-q_1}{2}}{\log \left(\frac{M-q_1}{2}\right)} < 0.961\frac{M-q_1}{\log (M-q_1)} < \frac{M}{\log M},\] where the inequalities follow from the fact that the function $\frac{X}{\log X}$ is an increasing function on the interval $(e, \infty)$. Therefore, using the symmetry of the choices $q_2$ and $q_3$, for a given $q_1$ and $M$, the number of choices of the tuple of primes $(q_2, q_3)$ is at most $2\frac{M}{\log M}$. Summing over all possible primes $q_1 \leq K_n$, and considering the $3$ symmetric permutations (i.e. instead of fixing $q_1$ we fix $q_2$ or $q_3$ and apply the same computation), we get that the number of ways $M$ can be written as a sum of three primes with at least one prime less than $K_n$ is at most \[3\sum_{\substack{q_1 \leq K_n \\ q_1 \, \text{prime}}} 2\frac{M}{\log M} < 6K_n\frac{M}{\log M},\] which completes the proof.
    \end{proof}
    
    Combining the above lemmas with \eqref{repn3primes}, it is enough to show that, for sufficiently large $M$, there exists at least one tuple of primes $(q_1, q_2, q_3)$ such that they are distinct and larger than $K_n$  (which is by definition $> \ell_{n-3}$) and $M = q_1 + q_2 + q_3$. We prove this in our next lemma. 
    
    \begin{lem}\label{thebound}
        For $n \geq 10^4$ and for every odd integer $M> K_n^2$, where $K_n$ is as defined in Proposition~\ref{alllarged}, we have \[\frac{M^2}{4(\log M)^3} > 1 + 6K_n\frac{M}{\log M} + 3\pi\left(\frac{M}{2}\right).\]
    \end{lem}
    
    \begin{proof}
        From~\cite[Corollary 1, eq. (3.6), p. 69]{RS}, we have, for $M> 4$, \[\pi\left(\frac{M}{2}\right) < 1.25506 \frac{M}{2(\log M - \log 2)} < 2\frac{M}{\log M}.\] Therefore it is enough to show that, for $n \geq 10^4$, \[\frac{M^{\frac{3}{2}}}{\log M}\left(\frac{M^{\frac{1}{2}}}{4(\log M)^2} - \frac{\log M}{M^{\frac{3}{2}}} - \frac{6K_n}{M^{\frac{1}{2}}}- \frac{6}{M^{\frac{1}{2}}}\right) > 0.\] Since, for $n \geq 10^4$, $M> K_n^2 = 2^{2n}\ell_0^2 > 4^{10^4}\ell_0^2 > 10^{10}$ (in particular $M > 2\cdot 10^{24}$, so that Lemma~\ref{smallprimes} applies), we have, for all $M>10^{10}$, \[\frac{\log M}{M^{\frac{3}{2}}} + \frac{6K_n}{M^{\frac{1}{2}}} + \frac{6}{M^{\frac{1}{2}}} < 1 + 6 + 6 = 13 < \frac{M^{\frac{1}{2}}}{4(\log M)^2},\] which concludes the proof.    
    \end{proof}
   \medskip
     \bibliographystyle{alpha}
    \bibliography{Bibliography}
    \end{document}